\documentclass[reqno,10pt]{amsart}

\usepackage{amsmath,amssymb,amsthm,mathrsfs}
\usepackage{enumerate}
\usepackage{graphicx}
\usepackage{wasysym}  
\usepackage{color}

\allowdisplaybreaks

\theoremstyle{plain}

\newtheorem{theorem}{Theorem}[section]
\newtheorem{lemma}[theorem]{Lemma}
\newtheorem{proposition}[theorem]{Proposition}
\newtheorem{corollary}[theorem]{Corollary}
\newtheorem{definition}[theorem]{Definition}
\newtheorem{ex}[theorem]{Example}

\numberwithin{equation}{section}

\DeclareMathOperator{\diag}{diag}

\newcommand{\R}{\ensuremath{\mathbb{R}}}
\newcommand{\Qset}{\ensuremath{\mathbb{Q}}}

\newcommand{\N}{\ensuremath{\mathbb{N}}}
\newcommand{\Pset}{\ensuremath{\mathbb{P}}}
\newcommand{\Xset}{\ensuremath{\mathbb{X}}}

\begin{document}

\title[Centrosymmetric and reverse matrices in bivariate orthogonal polynomials]{Centrosymmetric and reverse matrices in bivariate orthogonal polynomials}

\author[C. F. Bracciali]{Cleonice F. Bracciali}
\address{[C. F. Bracciali] Departamento de Matem\'{a}tica, IBILCE, UNESP - Universidade Estadual Paulista,
15054-000, S\~ao Jos\'e do Rio Preto, SP, Brazil. }
\email{cleonice.bracciali@unesp.br}

\author[G. S. Costa]{Glalco S. Costa} 
\address{[G. S. Costa] Departamento de Matem\'{a}tica, Instituto de Ci\^encias Tecnol\'ogicas e Exatas 
- ICTE, Universidade Federal do Tri\^{a}ngulo Mineiro - UFTM,
38025-180, Uberaba, MG, Brazil.}
\email{glalco.costa@uftm.edu.br}

\author[T. E. P\'{e}rez]{Teresa E. P\'erez}
\address{[T. E. P\'{e}rez] Instituto de Matem\'{a}ticas IMAG \&
	Departamento de Matem\'{a}tica Aplicada, Facultad de Ciencias. Universidad de Granada. 18071. Granada, Spain.}
\email{tperez@ugr.es}

\date{\today}

\begin{abstract}
We introduce the concept of reflexive moment functional in two variables and the definition of reflexive orthogonal polynomial system.
Also reverse matrices and their interesting algebraic properties are studied.
Reverse matrices and reflexive polynomial systems 
are directly connected in the context of bivariate orthogonal polynomials. Centrosymmetric matrices, reverse matrices and their connections with reflexive orthogonal polynomial systems are presented. Finally, several particular cases and examples are analysed.
\end{abstract}



\maketitle

\section{Introduction} 

When working with vectorial approach for multivariate orthogonal polynomials, it is necessary to make choices, such as the monomial order, which does not occur in the univariate case (see \cite{DX14}).
In addition to the choice of working with orthogonal, monic orthogonal or orthonormal polynomial vectors, the choice of a basis can simplify some theoretical objects, such as, coefficient matrices of three-term relations, structure relations, or explicit expressions for the polynomial vectors, among others. 
Within scope of the study of multivariate polynomials, many difficulties arise naturally with the increase of the dimension in relation to the univariate case. 
Then, it is important to simplify the objects' treatment as much as possible. 
The present work is inserted in this context. 

Following \cite{DX14}, we consider the linear space of real polynomials in two variables given by
$\Pi = \mathrm{span} \{ x^h\,y^k: h, k \geqslant 0 \}$
and the linear space
$\Pi_n = \mathrm{span} \{ x^h\,y^k: h+ k \leqslant n \}$
of finite dimension $(n+1)(n+2)/2$. 
A polynomial of total degree $n$ in two variables is a linear combination of monomials of total degree less than or equal to $n$. As usual, a two variable polynomial of degree $n$,   $p(x,y)\in \Pi_n$, is given by
$$
p(x,y) = \sum_{i+j\leqslant n} c_{i,j}\, x^i \, y^j, \quad c_{i,j}\in \mathbb{R}.
$$
We say that $p(x,y)$ is monic  if only has one term of highest degree in the form
$$
p(x,y) = x^{n-k}\,y^k + \sum_{i+j \leqslant n-1} \,c_{ij}\,x^{i}\,y^{j}, \quad c_{ij} \in \mathbb{R}.
$$

The vector representation for bivariate polynomials  by using the graded lexicographical order was introduced in \cite{Ko82a,Ko82b}, developed in \cite{Xu93} and it is the main representation used in the monography \cite{DX14}. 
A polynomial system is a sequence $\{\mathbb{P}_n\}_{n\geqslant0}$ of vectors of increasing size $n+1$ whose entries are polynomials, defined as 
$$
\mathbb{P}_n = (P_{n,0}^{n}(x,y), P_{n-1,1}^{n}(x,y), \ldots, P_{0,n}^{n}(x,y))^T, 
$$
such that the bivariate polynomial $P_{n-k,k}^{n}(x,y),$ $ k=0,1,\ldots,n,$ has exactly degree $n$. Also, for $n\geqslant 1$, the set of polynomial entries $\{P_{n,0}^{n}(x,y)$, $P_{n-1,1}^{n}(x,y)$, $\ldots,$ $P_{0,n}^{n}(x,y)\}$ is linearly independent. 

We say that a polynomial system is \emph{monic} if every polynomial entry, $P_{n-k,k}^{n}(x,y)$, is a monic bivariate polynomial, whose unique highest degree term is $x^{n-k}\,y^k$. 

\medskip

Let $\mathbf{u}$ be a moment linear functional defined on $\Pi$ by
$$
 \langle \mathbf{u}, x^my^n \rangle = \mu_{m,n}, \quad \mbox{for} \quad m,n \in \N \quad \mbox{and} \quad  \mu_{m,n} \in \R, 
$$
and extended by linearity.

Let $\{\Pset_n\}_{n\geqslant 0}$ be a polynomial system satisfying
\begin{align*}
&  \langle \mathbf{u},\mathbb{P}_n\,\mathbb{P}_n^T  \rangle = H_{n}, \\
&   \langle \mathbf{u},\mathbb{P}_n\,\mathbb{P}_m^T  \rangle = \mathtt{0}, 
\end{align*}
where $H_{n}$ is a non-singular symmetric positive definite matrix of order $n+1$ and  $\mathtt{0}$ is the zero matrix of adequate size. Then, 
$\{\Pset_n\}_{n\geqslant0}$ is called \emph{an orthogonal polynomial system} (OPS) associated with the moment functional $\mathbf{u}$.  A moment functional is called \emph{regular} or \emph{quasi-definite} if there exist orthogonal polynomial systems associated with it. Given a regular moment functional, there exist a unique monic orthogonal polynomial system (\cite{DX14}). Conditions about the existence of OPS associated with a given moment functional can be found in \cite{DX14}.

In the matrix approach for bivariate orthogonal polynomials, it is known (\cite{DX14}) that the orthogonal polynomial system satisfies the following three-term relations 
\begin{equation}\label{TTR-Intro}
\begin{aligned}
& x \Pset_n = \Lambda_{n,1}\Pset_{n+1} + \Gamma_{n,1}\Pset_n + \Upsilon_{n,1}\Pset_{n-1}, \\[0.5ex]
& y \Pset_n = \Lambda_{n,2}\Pset_{n+1} + \Gamma_{n,2}\Pset_n + \Upsilon_{n,2}\Pset_{n-1}, 
\end{aligned}
\end{equation}		
for $n \geqslant 0$, where  $\Pset_{-1}=0$, $\Pset_0=(1)$, $\Lambda_{n,i}$, 	$\Gamma_{n,i}$, and 	$\Upsilon_{n,i}$, for $i=1,2$ are matrices of adequate size.

A moment functional $\mathbf{u}$ is called \emph{centrally symmetric} if $\mu_{m,n} =0$ when $m+n$ is an odd integer. This condition is equivalent to the fact that $\Gamma_{n,i} = \mathtt{0}$, for $i=1,2$. Moreover, if $n$ is an even number, then every polynomial $P_{n-k,k}^{n}(x,y)$, for $0\leqslant k \leqslant n$, only contains monomials of total even degree, and if $n$ is an odd number, then 
only contains monomials of total odd degree.

\medskip

In this work we deal with a special condition of symmetry on the moment functional, which we call \emph{reflexive property}. This property can simplify many objects on the matrix approach, provide additional properties and facilities for the computational algorithms.  The \emph{reflexive property} for the moment functional implies a kind of symmetry into the polynomial vectors of the associated monic orthogonal polynomial systems, and induces a particular symmetric property (the \emph{reverse property}) for matrices, as we will study in Section 2. 
        
\medskip
 
The \emph{reflexive property} of a moment functional $\mathbf{u}$ is  defined by
$$ 
\langle \mathbf{u}, x^my^n \rangle = \langle \mathbf{u}, x^ny^m \rangle, \, \mbox{ for } m,n  \geqslant 0.
$$
For instance, we will show  that the
monic orthogonal polynomial system (MOPS) associated with $\mathbf{u},$ denoted by $\{\Qset_n\}_{n\geqslant 0}$, where 
$$
\Qset_n = (Q_{n,0}^{n}(x,y), Q_{n-1,1}^{n}(x,y), \ldots, Q_{0,n}^{n}(x,y))^T,
$$
satisfies  $Q_{n-k,k}^{n}(x,y) = Q_{k,n-k}^{n}(y,x)$, for $k=0,1,\ldots, n$. This property, that it will be called \emph{reflexive property} of the MOPS, is useful since only half of the calculation is needed to build the polynomial vectors, $\Qset_n,$ ${n \geqslant 0}$. As a consequence, the coefficient matrices of the three-term relations  \eqref{TTR-Intro}  have connections with each other. These connections provide a \emph{reverse property} among the associated matrices, and also a smaller number of calculation is necessary to construct them. 

Moreover, the moment functional has the \emph{reflexive property} if and only if the matrices involved in the three-term relations of the MOPS satisfy the \emph{reverse property}. 
In this work we are interesting to show how \emph{reverse matrices} and \emph{centrosymmetric matrices} are related with bivariate orthogonal polynomials.

\medskip

This work is organized into five sections. In the Section 2 the concepts of reverse matrices and centrosymmetric matrices as well as their most relevant properties are presented. Section 3 contains the basic theory of bivariate orthogonal polynomials using matrix approach, it also contains the definition and some properties of reflexive polynomial vectors. 
Reflexive moment functionals and the consequent properties of the associated bivariate orthogonal polynomial systems are introduced in Section 4. 
Section 5 brings the main focus of this work, the connection of \emph{reverse matrices}, \emph{centrosymmetric matrices} and \emph{reflexive bivariate orthogonal polynomials}. In Section 6 we present some non-trivial examples to illustrate these connections.

\section{Reverse matrices and centrosymmetric matrices} 

Centrosymmetric matrices satisfy several properties, including some algebraic structures. They have been extensively studied and many applications can be found in the literature, for instance see \cite{Ai56, An73, CB76, Ch98, Go70, WLL19, We85}.  

We start this section defining the reverse property between two matrices of size $(m+1) \times (n+1)$.

\begin{definition}
Consider two  matrices $X$ and $Y$ of size $(m+1) \times (n+1)$, denoted by $X=(x_{ij})_{i,j=0}^{m,n}$ and $Y=(y_{ij})_{i,j=0}^{m,n}$,  respectively. If
$$
	x_{ij}=y_{m-i,n-j}, \quad for \ i=0,1, \ldots, m, \ j=0,1, \ldots, n,
$$
then $X$ is called reverse matrix of $Y$ (conversely, $Y$ is reverse matrix of $X$). We will denote $X \leftrightharpoons Y$.
\end{definition}

\begin{ex}
Consider 
$X=\left(\begin{matrix}
		1 & 2 & 3 \\
		4 & 5 & 6
\end{matrix}\right)$ and 
$Y=\left(\begin{matrix}
		6 & 5 & 4 \\
		3 & 2 & 1 
\end{matrix}\right)$, 
then $X \leftrightharpoons Y$.	
\end{ex}

One can determinate the matrix $X^R$, the reverse matrix of $X=(x_{ij})_{i,j=0}^{m,n}$, by applying an operation called \emph{reflection}, i.e.,
\begin{equation} \label{reflection}
X^R = (x_{m-i,n-j})_{i,j=0}^{m,n}.
\end{equation}
Hence, $X \leftrightharpoons X^R.$

\medskip

The next result presents the minimum number of permutations between rows and between columns of a matrix $X$ needed to transform $X$ into its reverse matrix.
\begin{lemma} \label{permut.}
Let $X$ be a matrix of size $n$. Then, there exist $2\lfloor\frac{n}{2}\rfloor$ permutations between rows and between columns of $X$ such that transforms $X$ into $X^R$. 
\end{lemma}

\medskip

Now we present some direct properties of matrices that have the reverse property.

\begin{proposition} \label{pro_det}
Let  $X$ and $Y$ be square matrices such that $X\leftrightharpoons Y$, then $\det(X) = \det(Y)$.
\end{proposition}
\begin{proof}
Since an even number of permutations of rows or columns of a matrix preserves the value of  its determinant, the result follows from Lemma \ref{permut.}.
\end{proof}

For the next results, we denote by $X^{ij}, \, i=0,1,\ldots, m,$ $j=0,1,\ldots,n,$ the matrix obtained from a matrix $X$ of size $(m+1) \times (n+1)$ eliminating the row $i$ and the column $j$. The adjugate of the square matrix $X$, that is, the transposed matrix of cofactors of $X$,  is denoted by $adj X$, see \cite[p. 22]{HJ13}.

\begin{proposition} \label{properties X reverse Y} 
If $X$ and  $Y$ are matrices of size $(m+1) \times (n+1)$ satisfying $X \leftrightharpoons Y$, then
	
\noindent
1) $X^T \leftrightharpoons Y^T$.

\noindent
2) $XY^T \leftrightharpoons YX^T$.

\noindent
3)	$X^{ij} \leftrightharpoons Y^{m-i,n-j}$, for $i=0,1,\ldots,m,$ $ j=0,1,\ldots,n$.  	

\noindent
4)   If $m=n$, then $adj X \leftrightharpoons adj Y$.
\end{proposition}
\begin{proof} 
1) Consider  
$X=(x_{ij})_{i,j=0}^{m,n}$, 
$Y=(y_{ij})_{i,j=0}^{m,n}$, 
$X^T=(\widetilde{x}_{ij})_{i,j=0}^{n,m}$, and $Y^T=(\widetilde{y}_{ij})_{i,j=0}^{n,m}$ 
such that $\widetilde{x}_{ji} = x_{ij}$ and $\widetilde{y}_{ji} = y_{ij}$ for $i=0,1,\ldots,m,$ $j=0,1,\ldots,n$.
Since $x_{ij} = y_{m-i,n-j}$, for $j=0,1,\ldots,n$ and $i=0,1,\ldots,m,$ we have
$$
\widetilde{x}_{ji} = x_{ij} = y_{m-i, n-j} = \widetilde{y}_{n-j, m-i}.
$$
Therefore, $X^T \leftrightharpoons Y^T$.
	
2) Consider the same notation used in item 1). We know that
$XY^T$ is a matrix of size $(m+1)\times (m+1)$. Denoting $XY^T = (c_{ij})_{i,j=0}^{m}$ and 
$YX^T = (\widetilde{c}_{ij})_{i,j=0}^{m}$,
since $X \leftrightharpoons Y$, we see that
$$
c_{ij} = \sum_{k=0}^n x_{ik}\widetilde{y}_{kj}  = \sum_{k=0}^n y_{m-i,n-k} \widetilde{x}_{n-k,m-j} = 
\widetilde{c}_{m-i,m-j}.
$$
Hence, $XY^T \leftrightharpoons YX^T$.

3) It follows directly.

4)  Consider the matrices 
$X^*=(x^*_{ij})_{i,j=0}^{m,n}$ and $Y^* = (y^*_{ij})_{i,j=0}^{m,n}$ the matrices of cofactors of $X$ and $Y$ respectively.
From item 3) we know that $X^{ij} \leftrightharpoons Y^{n-i,n-j},$ for $ i,j = 0,1,\ldots,n$. The parity of $(i+j)$ is the same of $(n-i)+(n-j)$ then, from Proposition \ref{pro_det}, for $i,j=0,1,\ldots,n,$
$$ 
x^*_{ij} = (-1)^{i+j}\det (X^{ij}) = (-1)^{(n-i) + (n-j)} \det (Y^{n-i,n-j}) = y_{n-i, n-j}^*.
$$ 
Therefore, $X^* \leftrightharpoons Y^*$. Thus, from item 1), $adj X \leftrightharpoons adj Y$. 
\end{proof}

\medskip

There are many approaches for the centrosymmetric  property of a matrix, see \cite{Ai56, An73, CB76, HJ13, We85}. 

\medskip

For general matrices of size $(m+1) \times (n+1)$, one approach can be the following.
\begin{definition}	Consider a matrix $X=(x_{i,j})_{i,j=0}^{m,n}$ of size $(m+1) \times (n+1)$. The matrix  $X$ is centrosymmetric if it satisfies $X \leftrightharpoons X$, it means  $X^R=X$ and
$$
x_{i,j} = x_{m-i,n-j}, \quad for \ i=0,1,\ldots,m, \ j=0,1,\ldots,n.
$$	
\end{definition}

\begin{ex} 
	The matrices
	$$ A=  \begin{pmatrix} 
		1 & 2 & 3 & 4 & 5 \\
		6 & 7 & 8 & 7 & 6 \\
		5 & 4 & 3 & 2 & 1
	\end{pmatrix}, \quad 
	B= \begin{pmatrix}
		1 & 2 & 3 & 4 \\
		5 & 6 & 7 & 8 \\
		8 & 7 & 6 & 5 \\
		4 & 3 & 2 & 1
	\end{pmatrix}, \quad and \quad 
	C=  	\begin{pmatrix}
		1 & 2 & 3 \\
		4 & 5 & 4 \\
		3 & 2 & 1\end{pmatrix}$$
	are centrosymmetric matrices.
\end{ex}

In an approach for square matrices of size $(n+1)$, the \emph{reversal matrix} (or \emph{exchange matrix}), denoted by $J_{n+1}$, is a matrix that has ones along the secondary diagonal and zeros elsewhere.
Notice that $J_{n+1} = J_{n+1}^{-1}$.

A square matrix $X$ of size $(n+1)$ is a \emph{centrosymmetric matrix} if  $J_{n+1} X = XJ_{n+1}$.
Furthermore, a square matrix $X$ of size $(n+1)$ is called \emph{skew-centrosymmetric matrix} if 
$J_{n+1} X = - X J_{n+1}$, see
\cite{An73, CB76, HJ13}. 

A vector $V=(v_i)_{i=0}^{n}$ of size $(n+1)$ is called a \emph{symmetric vector} if $J_{n+1} V = V$, it means $v_i = v_{n-i}$, for $i = 0,1, \ldots, n$. Moreover, a vector $V=(v_i)_{i=0}^{n}$ is called a \emph{skew-symmetric vector} if $J_{n+1} V = - V$, it yields $v_i = -v_{n-i}$, $i = 0,1,\ldots, n,$  (\cite{CB76}). 

\medskip

The next proposition  brings some known direct consequences of the centrosymmetric property of a matrix.

\begin{proposition}
\label{prod. centrosymmetric matrices} 
~\\
\noindent
1) The sum of centrosymmetric matrices is a centrosymmetric matrix.

\noindent
2) The product of centrosymmetric matrices is a centrosymmetric matrix.

\noindent
3) The transpose of a centrosymmetric matrix is a centrosymmetric matrix.

\noindent
4) If $X$ is a square centrosymmetric matrix, then $adj X$ is also centrosymmetric.

\noindent
5) If $X$ is an invertible centrosymmetric matrix, then $X^{-1}$ is also  centrosymmetric.
\end{proposition}
\begin{proof}   
The proof of items 1), 2) and  3) follow directly from the definition of centrosymmetric matrix.
The proof of item 4) follows from  Proposition  \ref{properties X reverse Y} item 4) making $Y=X$. 

For completeness we include here a simple proof of item 5), using item 4). We denote $X^{-1} = (\widehat{x}_{ij})_{i,j=0}^{n}$ and  $adj X=(x^A_{ij})_{i,j=0}^{n}$. From item 4),  $x^A_{ij} = x^A_{n-i,n-j}$, for $i,j=0,1,\ldots,n$. Then, 
$$ 
\widehat{x}_{ij} = \dfrac{1}{\det X} x^A_{ij} = \dfrac{1}{\det X} x^A_{n-i, n-j} = \widehat{x}_{n-i,n-j},
$$
for $i,j=0,1,\ldots,n.$
\end{proof}

\medskip

The following result is a straightforward consequence of items 2) and  5)  of Proposition \ref{prod. centrosymmetric matrices}.

\begin{corollary}
If $XY$ is a centrosymmetric matrix and $X$ is an invertible centrosymmetric matrix, then $Y$ is also centrosymmetric matrix.
\end{corollary}

\medskip

Additionally, some basic properties of symmetric centrosymmetric matrices, including properties of their eigenvalues and eigenvectors,  were studied  in \cite{We85}. The next result, a consequence of \cite[Th.~11]{We85}, gives a characteristic of the eigenvectors of symmetric centrosymmetric positive definite matrices.

\begin{proposition}[\cite{We85}] \label{Prop_SCPD}
If $V$ is an eigenvector of a symmetric centrosymmetric positive definite matrix, then $V$ is either a symmetric vector or a skew-symmetric vector.
\end{proposition}

\medskip

Next result shows that one can construct, in different ways, centrosymmetric matrices from two matrices that have the reverse property.

\medskip

\begin{proposition}\label{prop_junType2} ~\\
\noindent
1) Let $T_1$ and $T_2$ be matrices of size $(m+1) \times (n+1)$.  $T_1 \leftrightharpoons T_2	$ if and only if, the matrix of size $2(m+1) \times 2(n+1)$
$$
\widehat{T} =	
\left( \begin{array}{c|c}
	T_1 & \mathtt{0} \\
				\hline
	\mathtt{0} & T_2
			\end{array} \right),
$$
where $\mathtt{0}$ is a zero matrix of size $(m+1)\times (n+1)$, is centrosymmetric. \\		
2) If $T_1$ and $T_2$ are matrices of size $(m+1)\times (n+1)$ such that $T_1 \leftrightharpoons T_2$, then $T_1 + T_2$ and  $T_1- T_2$  are centrosymmetric amtrices. 
\end{proposition}
\begin{proof} 
1) Suppose that $T_1 \leftrightharpoons T_2	$, hence, using the reflection operation \eqref{reflection}, we know that $T_1^R = T_2$ and  $T_2^R = T_1.$ Applying reflection operator in $\widehat{T}$ we get
$$
\widehat{T}^R =	
\left( \begin{array}{c|c}
 T_2^R & \mathtt{0} \\ \hline
 \mathtt{0} & T_1^R
\end{array} \right)
=	
\left( \begin{array}{c|c}
 T_1 & \mathtt{0} \\ \hline
 \mathtt{0} & T_2
\end{array} \right) = \widehat{T}.
$$
Therefore,  $\widehat{T}$ is a centrosymmetric matrix.

Reciprocally, supposing that  $\widehat{T}^R =\widehat{T}$,   we see  that $T_2^R = T_1$ and  $T_1^R = T_2,$ hence $T_1 \leftrightharpoons T_2$.
		
2) The proof follows directly from the fact that $T_1^R = T_2$ and  $T_2^R = T_1.$
\end{proof}

\section{Bivariate orthogonal polynomials}

We start  recalling the basic definitions and main tools about bivariate orthogonal polynomials  that we will need in the rest of the paper. We refer mainly \cite{DX14}.

As we mention on the introduction, a \emph{polynomial system (PS)}  is a sequence of polynomial vectors  $\{\mathbb{P}_n\}_{n\geqslant0}$ of increasing size $n+1$,  
$$
\mathbb{P}_n = (P_{n,0}^{n}(x,y), P_{n-1,1}^{n}(x,y), \ldots, P_{0,n}^{n}(x,y))^T,
$$
where the bivariate polynomial $P_{n-k,k}^{n}(x,y),$ $ k=0,1,\ldots,n,$ has exactly degree $n$, and the set of polynomial entries is linearly independent. 

The simplest polynomial system is the so-called \emph{canonical basis} $\{\mathbb{X}_n\}_{n\geqslant0}$, 
defined as 
$$
\mathbb{X}_n = (x^n, x^{n-1}\,y, x^{n-2}\,y^2, \ldots, x\,y^{n-1}, y^{n})^T.
$$

If $\{\mathbb{P}_n\}_{n\geqslant0}$ is a polynomial system, then there exist real matrices  $G^n_k$ of size $(n+1)\times (k+1)$ such that every polynomial vector $\mathbb{P}_n$ can be expressed in terms of the canonical basis as
\begin{equation} \label{expl.exp}
	\mathbb{P}_n = G_n \,\mathbb{X}_n + G_{n-1}^{n} \,\mathbb{X}_{n-1} +
	G_{n-2}^{n}\, \mathbb{X}_{n-2} + \cdots  + G_{1}^{n}\, \mathbb{X}_{1}  + G_{0}^{n} \, \mathbb{X}_{0}.
\end{equation}
In the particular case when $G_n = I_{n+1}$, $\{\Pset_n\}_{n\geqslant 0}$ is called \emph{monic polynomial system}.

\medskip

A polynomial system can satisfy a symmetric property on the entries of the polynomial vectors.

\begin{definition}
Let $\Pset_n = (P_{n,0}^{n}(x,y), P_{n-1,1}^{n}(x,y), \ldots, P_{0,n}^{n}(x,y))^T$ be a polynomial vector satisfying 
\begin{equation*}
P_{n-k,k}^{n}(x,y) = P_{k,n-k}^{n}(y,x), \quad k=0,1,\ldots,n,
\end{equation*}
then $\Pset_n$ is called reflexive polynomial vector.
A polynomial system such that all the polynomial vectors are reflexives will be called  reflexive polynomial system.
\end{definition}

We remark that in the univariate case, a reflexive property of a polynomial is defined according with its coefficients. First consider a complex polynomial in one variable
$r(x) = b_n x^n + b_{n-1} x^{n-1} + \cdots + b_0$,  $b_i \in \mathbb{C}$, the polynomial 
$
r^*(x) = x^n  \overline{r \left(1/\overline{x}\right)} = \overline{b}_0 x^n + \overline{b}_1 x^{n-1} + \cdots + \overline{b}_n,
$ is called \emph{reciprocal polynomial} of $r(x)$. If $r^*(x)= u r(x)$, with $|u|=1,$ then $r(x)$ is called \emph{self-inversive polynomial}, see \cite{MiMiRa19}. These type of polynomials appear also on the theory of orthogonal polynomials on the unit circle.

Consider now a real polynomial
$r(x) = b_n x^n + b_{n-1} x^{n-1} + \cdots + b_0,$  $b_i \in \mathbb{R}$, and the polynomial
$x^n  r \left(1/x\right) = b_0 x^n + b_1 x^{n-1} + \cdots + b_n.$
If $r(x) = x^n r(1/x)$, it means $ b_i = b_{n-i},$ for $i=0,1,\ldots, n$, then $r(x)$ is called \emph{self-reciprocal polynomial} or \emph{palindromic polynomial}. In several works, for instance in \cite{BoBrPe16, BoMaMe14, KiPa08, Kw11}, the behaviour of the zeros of special cases of this type of univariate polynomials was studied. 

\medskip

Next result shows a relation between reflexive polynomial vectors and centrosymmetric matrices. More precisely, it says that the matrix of change of basis between two reflexive polynomial vectors is necessarily centrosymmetric. Reciprocally,  the product of centrosymmetric matrix and reflexive polynomial vector preserves the reflexive property of the polynomial vector.  

\begin{proposition} \label{mud. de base}
Let $\{\Pset_n\}_{n\geqslant 0}$ and $\{\widetilde{\Pset}_n\}_{n\geqslant 0}$ be two polynomial systems such that $\Pset_n$ is a reflexive polynomial and  $\widetilde{\Pset}_n = T_n \Pset_n$, for $n \in \N$. Then, $\widetilde{\Pset}_n$ is reflexive polynomial vector if and only if $T_n$ is centrosymmetric matrix.
\end{proposition}
\begin{proof}
Let $\widetilde{\Pset}_n = T_n \Pset_n$ be written as	
	$$ \left(\begin{array}{c}
		\widetilde{P}_{n,0}^{n}(x,y)\\
		\vdots \\
		\widetilde{P}_{n-k,k}^{n}(x,y)\\
		\vdots \\
		\widetilde{P}_{0,n}^{n}(x,y)
	\end{array} \right) = \left( \begin{matrix}
	t_{00}^{(n)} & t_{01}^{(n)} & \ldots & t_{0n}^{(n)} \\
	\vdots   & \vdots   &  & \vdots   \\
	t_{k0}^{(n)} & t_{k1}^{(n)} & \ldots & t_{kn}^{(n)} \\
	\vdots   & \vdots   &   & \vdots   \\
	t_{n0}^{(n)} & t_{n1}^{(n)} & \ldots & t_{nn}^{(n)}
\end{matrix}\right) \left(\begin{array}{c}
P_{n,0}^{n}(x,y)\\
\vdots \\
P_{n-k,k}^{n}(x,y)\\
\vdots \\
P_{0,n}^{n}(x,y)
\end{array} \right).$$
Since $\Pset_n$ is reflexive polynomial vector, 
\begin{equation}\label{tilde P_n-k}
	\widetilde{P}_{n-k,k}^{n}(x,y) = \sum_{j=0}^n t_{kj}^{(n)} P_{n-j,j}^{n}(x,y) = \sum_{j=0}^n t_{kj}^{(n)} P_{j,n-j}^{n}(y,x).
\end{equation}	
On the other hand, we can write $\widetilde{P}_{k,n-k}^{n}(y,x)$ as
\begin{equation} \label{tilde P_k}
\widetilde{P}_{k,n-k}^{n}(y,x) = \sum_{j=0}^n t_{n-k,j}^{(n)} P_{n-j,j}^{n}(y,x).
\end{equation}
Thus, $	\widetilde{P}_{n-k,k}^{n}(x,y) = \widetilde{P}_{k,n-k}^{n}(y,x)$ implies that
\begin{equation*}
 \sum_{j=0}^n t_{kj}^{(n)} P_{j,n-j}^{n}(y,x) = \sum_{j=0}^n t_{n-k,j}^{(n)} P_{n-j,j}^{n}(y,x).  
\end{equation*}  
Making  $l = n-j$ in the second summation, we have
\begin{equation*}
	  \sum_{j=0}^n t_{kj}^{(n)} P_{j,n-j}^{n}(y,x) = \sum_{l=0}^n t_{n-k,n-l}^{(n)} P_{l,n-l}^{n}(y,x).  
\end{equation*} 
Hence, 
\begin{equation*}
\sum_{j=0}^n [t_{k,j}^{(n)} - t_{n-k,n-j}^{(n)}] P_{j,n-j}^{n}(y,x) = 0,
\end{equation*}
and since $\{P_{n-j,j}^{n}\}$ is linearly independent, it follows that 
$ t_{kj}^{(n)} = t_{n-k,n-j}^{(n)}, $ $ j=0,1,\ldots,n.$
That is, $T_n$ is centrosymmetric matrix.

Conversely, suppose that $T_n$ is centrosymmetric matrix and $\Pset_n$ is reflexive polynomial vector, from \eqref{tilde P_n-k} we know that
$$	
\widetilde{P}_{n-k,k}^{n}(x,y) = \sum_{j=0}^n t_{kj}^{(n)} P_{n-j,j}^{n}(x,y) = \sum_{j=0}^n t_{n-k,n-j}^{(n)} P_{j,n-j}^{n}(y,x).$$
Making $l=n-j$ on the latter summation, from \eqref{tilde P_k}, we have
$$ 
\widetilde{P}_{n-k,k}^{n}(x,y)  
= \sum_{j=0}^n t_{n-k,l}^{(n)} P_{n-l,l}^{n}(y,x) = \widetilde{P}_{k,n-k}^{n}(y,x).$$
Then, 
$\widetilde{\Pset}_n $ is reflexive polynomial vector. 
\end{proof}

\medskip
	
 From now on, we will work with bivariate moment functionals $\mathbf{u}$ defined as
\begin{equation*} 
\langle \mathbf{u}, f \rangle = \iint_{\Omega} f(x,y)\, W(x,y)\, dx\,dy,
\end{equation*}
where $W(x,y)$ is a weight function defined on a region $\Omega\subset\R^2$ such that its associated moments satisfy
$$
\mu_{n,m} =\langle \mathbf{u}, x^n\,y^m \rangle = \iint_{\Omega} x^n\,y^m\, W(x,y)\, dx\,dy < +\infty,
$$
for $n, m = 0, 1, \ldots$. 
Thus, we are considering the inner product
\begin{equation} \label{ip}
		(f, g) := \langle \mathbf{u}, f\,g\rangle = \iint_{\Omega} f(x,y)\,g(x,y)\,W(x,y)\, dx\,dy.
\end{equation}
		
\medskip
	
In terms of the inner product \eqref{ip}, 	if $\{\Pset_n\}_{n\geqslant 0}$ is a polynomial system satisfying
\begin{align*}
		&  (\mathbb{P}_n, \mathbb{P}_n^T) = \langle \mathbf{u},\mathbb{P}_n\,\mathbb{P}_n^T  \rangle = H_{n}, \\
		&  (\mathbb{P}_n, \mathbb{P}_m^T) = \langle \mathbf{u},\mathbb{P}_n\,\mathbb{P}_m^T  \rangle = \mathtt{0}, 
\end{align*}
where $H_{n}$ is non-singular symmetric positive definite matrix of size $n+1$ and  $\mathtt{0}$ is zero matrix of adequate size. Then, $\{\Pset_n\}_{n\geqslant 0}$  is an \emph{orthogonal polynomial system} (OPS) associated with the weight function $W(x,y).$ If $H_n = I_{n+1}$, the identity matrix of order $n+1$, then $\{\Pset_n\}_{n\geqslant 0}$ is called \emph{orthonormal polynomial system}, and  they are not unique. As we recall at the introduction, there is a unique \emph{monic orthogonal polynomial system} associated with the inner product $(\cdot, \cdot)$. In the sequel, we denote  by $\{\mathbb{P}_n\}_{n\geqslant 0}$ an orthonormal polynomial system, and by $\{\Qset_n\}_{n\geqslant 0}$ the monic orthogonal polynomial system.

\medskip
	
Orthogonal polynomial systems satisfy two three-term relations, one for each variable, as \eqref{TTR-Intro}. For orthonormal polynomial system, $\{\Pset_n\}_{n\geqslant 0}$, and monic orthogonal polynomial system, $\{\Qset_n\}_{n\geqslant 0}$, the three-term relations have, respectively, the following forms
	\begin{equation}\label{O3TR}
		x_i\Pset_n = A_{n,i}\Pset_{n+1} + B_{n,i}\Pset_n + A^T_{n-1,i}\Pset_{n-1}, \quad i=1,2,
	\end{equation}
for $n\geqslant 0$, $\mathbb{P}_{-1} =0$ and $\mathbb{P}_0 = (\mu_{0,0}^{-1/2})$, and
	\begin{equation}\label{M3TR}
		x_i\Qset_n = L_{n,i}\Qset_{n+1} + C_{n,i}\Qset_n + D_{n,i}\Qset_{n-1}, \quad i=1,2,
	\end{equation}
for $n\geqslant 0$, $\mathbb{Q}_{-1} =0$ and $\mathbb{Q}_0 = (1)$.  Here $x_1=x$, $x_2=y$, in order to simplify the notation. The matrices $A_{n,i}$  of size $(n+1) \times (n+2)$ are full rank matrices, $B_{n,i}$ and $C_{n,i}$ are matrices of size $(n+1) \times (n+1)$, $D_{n,i}$ are full rank matrices of size $(n+1)\times n$, and $L_{n,i}$ are matrices of size $(n+1)\times(n+2)$ such that $x_i\Xset_n=L_{n,i}\Xset_{n+1}, $ $ i=1,2$, that are
$$L_{n,1} = \left(\begin{array}{cccc|c}
	1        &      &  & \bigcirc & 0\\
	         &   1    &  &  & 0\\
	 &    & \ddots &  & \vdots\\
	\bigcirc &   &  & 1 & 0\\
\end{array}\right)
\quad \mbox{and} \quad 
L_{n,2} = \left(\begin{array}{c|cccc}
	0 & 1 &    &  & \bigcirc \\
	0 &   &  1 &  &  \\
	\vdots &   &    & \ddots &  \\
	0 & \bigcirc &   &  & 1 \\
\end{array}\right).$$

\medskip

Given $\mathbf{u}$ a moment functional,  the \emph{matrix of moments} $M_n$,  $n\in\N$, is a matrix of  size $(n+1)(n+2)/2$ defined by 
\begin{equation} \label{Matrix_M_n}
M_{n}=\left( \begin{matrix}
			\langle \mathbf{u}, \Xset_0 \Xset_0^T \rangle & \langle \mathbf{u}, \Xset_0 \Xset_1^T \rangle & \cdots & \langle \mathbf{u}, \Xset_0 \Xset_n^T \rangle\\
			\langle \mathbf{u}, \Xset_1 \Xset_0^T  \rangle & \langle \mathbf{u}, \Xset_1 \Xset_1^T \rangle & \cdots & \langle \mathbf{u}, \Xset_1 \Xset_n^T \rangle \\
			\vdots & \vdots &   & \vdots \\
			\langle \mathbf{u}, \Xset_n \Xset_0^T \rangle & \langle \mathbf{u}, \Xset_n \Xset_1^T \rangle & \cdots & \langle \mathbf{u}, \Xset_n \Xset_n^T \rangle
		\end{matrix} \right),
	\end{equation}
where $\langle \mathbf{u}, \Xset_m \Xset_n^T \rangle$ is a block of size $(m+1) \times (n+1)$ given by
\begin{equation}  \label{Part_Matrix_M_n}
\langle \mathbf{u}, \Xset_m \Xset_n^T \rangle 
=  \left(\begin{matrix}
		\langle \mathbf{u}, x^{m+n} \rangle & \langle \mathbf{u}, x^{m+n-1}y \rangle  & \ldots &  \langle \mathbf{u}, x^my^n \rangle \\
		\langle \mathbf{u}, x^{m+n-1}y \rangle  & \langle \mathbf{u}, x^{m+n-2}y^2 \rangle & \ldots & \langle \mathbf{u}, x^{m-1}y^{n+1} \rangle \\
		\vdots & \vdots &        & \vdots \\
		\langle \mathbf{u}, x^ny^m \rangle & \langle \mathbf{u}, x^{n-1}y^{m+1} \rangle & \ldots & \langle \mathbf{u}, y^{m+n} \rangle
	\end{matrix} \right).
\end{equation}
See \cite{DX14} for more details.	

	Let $\alpha = (\alpha_1, \alpha_2) $ be a bi-index such that $ |\alpha| = \alpha_1 + \alpha_2 = n $, it means that   $\alpha = (n-k, k)$, for $k=0,1,\ldots, n$. Using the notation $\mathbf{x}^\alpha = x^{\alpha_1} y^{\alpha_2}$, the matrix $M_{\alpha} (x,y) $ is defined by
\begin{equation} \label{matrixMalpha}
M_{\alpha}(x,y) = \left( 
\begin{tabular}{c|c} 
			& $\langle \mathbf{u}, \mathbf{x}^\alpha \Xset_0 \rangle$ \\
			$M_{n-1}$ & $\langle \mathbf{u}, \mathbf{x}^\alpha \Xset_1 \rangle$ \\
			& $\vdots$\\
			& $\langle \mathbf{u}, \mathbf{x}^\alpha \Xset_{n-1} \rangle$ \\
			\hline
			$\Xset_0^T \quad \Xset_1^T \quad  ... \quad \Xset_{n-1}^T$ &  $\textbf{x}^\alpha$
		\end{tabular} \right).
\end{equation}
	
\begin{proposition}[\cite{DX14}] \label{construction of pol}
Consider the monic polynomials $Q^n_{n-k,k}(x,y)$,  for $k=0,1,\ldots,n$, the entries of $\mathbb{Q}_n = (Q_{n,0}^{n}(x,y), Q_{n-1,1}^{n}(x,y), \ldots, Q_{0,n}^{n}(x,y))^T,$ given by
\begin{equation*} 
Q^n_{n-k,k}(x,y) = \dfrac{\det(M_{(n-k,k)}(x,y))}
{\det(M_{n-1})},
\quad   k=0,1,\ldots,n, 
\end{equation*}
then $\{\Qset_n\}_{n\geqslant 0}$ is the MOPS with respect to the moment functional $\mathbf{u}$.\end{proposition}

\section{Reflexive Bivariate Orthogonal Polynomials} 

First we need the definition of \emph{reflexive weight function}.
\begin{definition}
Consider a region $\Omega \subseteq \R^2$, such that it satisfies $(x,y) \in \Omega \Leftrightarrow (y,x) \in \Omega.$ A weight function $W(x,y)$ defined in $\Omega$, satisfying
	\begin{equation} \label{W(x,y)=W(y,x)}
		W(x,y) = W(y,x), \quad (x,y) \in \Omega,
	\end{equation}
	is called a  reflexive weight function.
\end{definition}

As we define at the introduction, a \emph{reflexive moment functional} is a functional $\mathbf{u}$ such that its associated moments satisfy 
$$
\mu_{m,n} = 
\mu_{n,m}, \quad m,n=0,1,\ldots .
$$
In particular, if $W(x,y)$ is a \emph{reflexive weight function}, the moment functional 
\begin{equation*}
	\langle \mathbf{u}, f \rangle = \iint_{\Omega} f(x,y)\, W(x,y)\, dx\,dy,
\end{equation*}
is a \emph{reflexive moment functional}.

Now we can show that a reflexive weight function yields reflexive MOPS. This is one of the main results involving reflexive weight function. 

\begin{theorem} \label{prop_sim_pol} 
The polynomial vectors of the MOPS $\{\Qset_n\}_{n\geqslant 0}$ associated with a reflexive weight function are reflexives, it means
\begin{equation*} 
Q_{n-k,k}^{n}(x,y) = Q_{k,n-k}^{n}(y,x), \quad k=0,1 \ldots,n.
\end{equation*}
\end{theorem}
\begin{proof}
From Proposition \ref{construction of pol}, we can write, for $k=0,1, \ldots,n,$ 
\begin{equation*} 
Q_{n-k,k}^{n}(x,y) = \dfrac{\det(M_{(n-k,k)}(x,y))}
{\det(M_{n-1}(x,y))}		
\quad \mbox{and} \quad Q_{k,n-k}^{n}(y,x) = \dfrac{\det(M_{(k,n-k)}(y,x))}
{\det(M_{n-1}(y,x))},
\end{equation*}
 where, from  \eqref{Matrix_M_n} and \eqref{Part_Matrix_M_n},  the matrices $M_{n-1}(x,y)$ and $M_{n-1}(y,x)$ are 
\begin{equation*}
M_{n-1}(x,y) =
\left(\begin{matrix}
\mu_{00} & \mu_{10} & \mu_{01} & \cdots & \mu_{n-1,0} & \cdots & \mu_{0,n-1}   \\[1ex]
\mu_{10} & \mu_{20} & \mu_{11} & \cdots & \mu_{n0} & \cdots & \mu_{1,n-1}    \\[1ex]
\mu_{01} & \mu_{11} &  \mu_{02}   & \cdots & \mu_{n-1,1} & \cdots & \mu_{0n} \\
\vdots  & \vdots & \vdots &  & \vdots & & \vdots    \\
 \mu_{n-1,0} & \mu_{n,0} & \mu_{n-1,1} & \cdots & \mu_{2n-2,0} & \cdots & \mu_{n-1,n-1}   \\
\vdots & \vdots & \vdots &  & \vdots & & \vdots    \\
\mu_{0,n-1} & \mu_{1,n-1} & \mu_{0,n} & \cdots & \mu_{n-1,n-1} & \cdots & \mu_{0,2n-2}   \\[1ex]
\end{matrix}\right)  	
\end{equation*}
and
\begin{equation*}
M_{n-1}(y,x)=  
\left(\begin{matrix}
\mu_{00} & \mu_{01} & \mu_{10} & \cdots & \mu_{0,n-1} & \cdots & \mu_{n-1,0}  \\[1ex]
\mu_{01} & \mu_{02} & \mu_{11} & \cdots & \mu_{0n} & \cdots & \mu_{n-1,1}      \\[1ex]
\mu_{10} & \mu_{11} &  \mu_{20}   & \cdots & \mu_{1,n-1} & \cdots & \mu_{n,0}   \\
\vdots  & \vdots & \vdots &  & \vdots & & \vdots   \\
 \mu_{0,n-1} & \mu_{0,n} & \mu_{1,n-1} & \cdots & \mu_{0,2n-2} & \cdots & \mu_{n-1,n-1}   \\
\vdots & \vdots & \vdots &  & \vdots & & \vdots    \\
\mu_{n-1,0} & \mu_{n-1,1} & \mu_{n,0} & \cdots & \mu_{n-1,n-1} & \cdots & \mu_{2n-2,0}   \\[1ex]
\end{matrix}\right).  	
\end{equation*}
Since the weight function is reflexive, i.e., $\mu_{m,n} =  \mu_{n,m},$ for $ m,n=0,1,\ldots ,$
 we observe that $M_{n-1}(x,y) = M_{n-1}(y,x)$ and their determinants are the same. \\

Now, from \eqref{matrixMalpha} the matrices $M_{(n-k,k)}(x,y)$ and $ M_{(k,n-k)}(y,x)$ are given, respectively, by
\begin{equation*}
\left(\begin{matrix}
\mu_{00} & \mu_{10} & \mu_{01} & \cdots & \mu_{n-1,0} & \cdots & \mu_{0,n-1} & \mu_{n-k,k}  \\[1ex]
\mu_{10} & \mu_{20} & \mu_{11} & \cdots & \mu_{n0} & \cdots & \mu_{1,n-1}  & \mu_{n-k+1,k}   \\[1ex]
\mu_{01} & \mu_{11} &  \mu_{02}   & \cdots & \mu_{n-1,1} & \cdots & \mu_{0n} &  \mu_{n-k,k+1} \\
\vdots  & \vdots & \vdots &  & \vdots & & \vdots & \vdots  \\
 \mu_{n-1,0} & \mu_{n,0} & \mu_{n-1,1} & \cdots & \mu_{2n-2,0} & \cdots & \mu_{n-1,n-1} &  \mu_{2n-k-1,k} \\
\vdots & \vdots & \vdots &  & \vdots & & \vdots & \vdots  \\
\mu_{0,n-1} & \mu_{1,n-1} & \mu_{0,n} & \cdots & \mu_{n-1,n-1} & \cdots & \mu_{0,2n-2} & \mu_{n-k,k+n-1} \\[1ex]
1  & 	x	 & y & \cdots &  x^{n-1}  & \cdots &   y^{n-1}  & x^{n-k}y^k 		\end{matrix}\right)  	
\end{equation*}
and
\begin{equation*}
\left(\begin{matrix}
\mu_{00} & \mu_{01} & \mu_{10} & \cdots & \mu_{0,n-1} & \cdots & \mu_{n-1,0} & \mu_{n-k,k}  \\[1ex]
\mu_{01} & \mu_{02} & \mu_{11} & \cdots & \mu_{0n} & \cdots & \mu_{n-1,1}  & \mu_{n-k,k+1}   \\[1ex]
\mu_{10} & \mu_{11} &  \mu_{20}   & \cdots & \mu_{1,n-1} & \cdots & \mu_{n,0} &  \mu_{n-k+1,k} \\
\vdots  & \vdots & \vdots &  & \vdots & & \vdots & \vdots  \\
 \mu_{0,n-1} & \mu_{0,n} & \mu_{1,n-1} & \cdots & \mu_{0,2n-2} & \cdots & \mu_{n-1,n-1} &  \mu_{n-k,k+n-1} \\
\vdots & \vdots & \vdots &  & \vdots & & \vdots & \vdots  \\
\mu_{n-1,0} & \mu_{n-1,1} & \mu_{n,0} & \cdots & \mu_{n-1,n-1} & \cdots & \mu_{2n-2,0} & \mu_{2n-k-1,k} \\[1ex]
1  & y	 & x & \cdots &  y^{n-1}  & \cdots &   x^{n-1}  & x^{n-k}y^k 		\end{matrix}\right).  	
\end{equation*}

Once more using the fact that $\mu_{m,n} =  \mu_{n,m},$ for $ m,n=0,1,\ldots ,$ we see that
 the number of permutations of rows and permutations of columns that transforms the matrix $M_{(n-k,k)}(x,y)$ into the matrix $ M_{(k,n-k)}(y,x)$ is even, hence 
\begin{equation*} 
\det (M_{n-k,k}(x,y)) = \det (M_{k,n-k}(y,x)), \quad for  \ k=0,1 \ldots,n,
\end{equation*}
and the result holds.
\end{proof}
 
When the polynomial vectors of the MOPS $\{\Qset_n\}_{n\geqslant 0}$ are reflexives for $n\geqslant0$, then $\{\Qset_n\}_{n\geqslant 0}$ is called reflexive MOPS. 

Next results give some connections of reflexive MOPS and some centrosymmetric matrices.

\begin{corollary} \label{Hn is self}
	Let $\{\Qset_n\}_{n\geqslant 0}$ be a MOPS associated with a  reflexive weight function satisfying \eqref{W(x,y)=W(y,x)}. Then, for $n\in\N$, the matrix $H_n$ of size $n+1$, given by 
$H_n = \langle \mathbf{u}, \Qset_n\Qset_n^T \rangle,$ is centrosymmetric, i.e.,  $H_n = (h_{ij})_{i,j=0}^{n},$
satisfies		
$$	
h_{ij} = h_{n-i,n-j}, \quad \mbox{for} \ i,j=0,1, \ldots,n.
$$
\end{corollary}		
\begin{proof}	
Observe that $h_{ij} = \langle \mathbf{u}, Q_{n-i,i}^{n}(x,y) Q_{n-j,j}^{n}(x,y) \rangle $, for $i,j=0,1,\ldots,n$.

From Theorem \ref{prop_sim_pol}, we can write
	\begin{eqnarray*}
		h_{ij}  =  \displaystyle \iint_{\Omega} Q_{i,n-i}^{n}(y,x) Q_{j,n-j}^{n}(y,x) W(x,y) dxdy.
	\end{eqnarray*}
Making a change of variables $x\leftrightarrow y$, we have
\begin{eqnarray*}
h_{ij} = \displaystyle \iint_{\Omega} Q_{i,n-i}^{n}(x,y) Q_{j,n-j}^{n}(x,y) W(y,x) dxdy. 
\end{eqnarray*}
Finally, using the property \eqref{W(x,y)=W(y,x)}, we obtain
\begin{eqnarray*}
h_{ij} = \displaystyle \iint_{\Omega} Q_{i,n-i}^{n}(x,y) Q_{j,n-j}^{n}(x,y) W(x,y) dxdy =  h_{n-i,n-j}.
\end{eqnarray*}	
\end{proof}		

Consider $\{\Qset_{n}\}_{n\geqslant 0}$ the MOPS associated with a weight function $W(x,y)$, and  $\{\Pset_n\}_{n\geqslant 0}$ another OPS also associated with $W(x,y)$. Using the explicit expression (\ref{expl.exp}) and the orthogonality, the polynomial vectors $\Pset_n$ and 
$\Qset_n$ are related by   $\Pset_n = G_n \Qset_{n},$ $ n \geqslant 0$. In particular, as a consequence of Proposition \ref{mud. de base}, we have the following result.
 
\begin{corollary}
 If $\{\Qset_n\}_{n\geqslant 0}$ is the MOPS associated with a reflexive weight function $W(x,y)$ and $\{\Pset_n\}_{n\geqslant 0}$ is an OPS associated with $W(x,y)$ and given by $\Pset_n = G_n \Qset_{n}$, then $\Pset_n$ is reflexive polynomial vector if and only if $G_n$ is a centrosymmetric matrix.  
\end{corollary}

\section{Main Results}

We now present more connections involving reverse matrices, centrosymmetric matrices and reflexive bivariate orthogonal polynomial systems.

The first result shows that the coefficient matrices of the three-term relations for reflexive MOPS given by \eqref{M3TR} have the reverse property. 

\begin{theorem} \label{sobre_E_1_E_2}
Let $\{\Qset_n\}_{n\geqslant 0}$ be a MOPS associated with a reflexive weight function satisfying \eqref{W(x,y)=W(y,x)} and let \eqref{M3TR} be their three-term relations. Then, for $n\in \N$, $C_{n,1} \leftrightharpoons C_{n,2}$, and $D_{n,1} \leftrightharpoons D_{n,2}$. 
\end{theorem}
\begin{proof}	
First we denote the matrices	$C_{n,k}=(c^{(k)}_{ij})_{i,j=0}^{n,n}$ and $D_{n,k}=(d^{(k)}_{ij})_{i,j=0}^{n,n-1}$, for  $k=1,2.$ 
From the three-term relation \eqref{M3TR}, omitting the variables $(x,y)$, we can write	
\begin{equation}\label{3trx}
		x \left(\begin{array}{c}
			Q_{n,0}^n\\
			\vdots \\
			Q_{n-k,k}^n\\
			\vdots \\
			Q_{0,n}^n
		\end{array} \right) = \left(\begin{array}{c}
			Q_{n+1,0}^{n+1}\\
			\vdots \\
			Q_{n+1-k,k}^{n+1}\\
			\vdots \\
			Q_{1,n}^{n+1}
		\end{array} \right) + C_{n,1} \left(\begin{array}{c}
			Q_{n,0}^{n}\\
			\vdots \\
			Q_{n-l,l}^{n}\\
			\vdots \\
			Q_{0,n}^{n}
		\end{array} \right) + D_{n,1} \left(\begin{array}{c}
			Q_{n-1,0}^{n-1}\\
			\vdots \\
			Q_{n-1-s,s}^{n-1}\\
			\vdots \\
			Q_{0,n-1}^{n-1}
		\end{array} \right)
	\end{equation}	
and
	\begin{equation}\label{3try}
		y \left(\begin{array}{c}
			Q_{n,0}^n\\
			\vdots \\
			Q_{k,n-k}^n\\
			\vdots \\
			Q_{0,n}^n
		\end{array} \right) = \left(\begin{array}{c}
			Q_{n,1}^{n+1}\\
			\vdots \\
			Q_{k,n+1-k}^{n+1}\\
			\vdots \\
			Q_{0,n+1}^{n+1}
		\end{array} \right) + C_{n,2} \left(\begin{array}{c}
			Q_{n,0}^{n}\\
			\vdots \\
			Q_{n-l,l}^{n}\\
			\vdots \\
			Q_{0,n}^{n}
		\end{array} \right) + D_{n,2} \left(\begin{array}{c}
			Q_{n-1,0}^{n-1}\\
			\vdots \\
			Q_{n-1-s,s}^{n-1}\\
			\vdots \\
			Q_{0,n-1}^{n-1}
		\end{array} \right).
	\end{equation}	
		
From \eqref{3trx},  for $k = 0,1,\ldots,n$, we can write $xQ_{n-k,k}^n(x,y)$ as
\begin{eqnarray} \label{xQ}
& & xQ_{n-k,k}^n(x,y) = Q_{n+1-k,k}^{n+1}(x,y) \nonumber \\[1ex]
& &  + c^{(1)}_{k0}Q_{n,0}^{n}(x,y) + \cdots + c^{(1)}_{kl} Q_{n-l,l}^{n}(x,y) + \cdots + c^{(1)}_{kn}Q_{0,n}^{n}(x,y) \\[1ex]
& &   + d^{(1)}_{k0}Q_{n-1,0}^{n-1}(x,y) + \cdots + d^{(1)}_{ks} Q_{n-1-s,s}^{n-1}(x,y) + \cdots + d^{(1)}_{k,n-1} Q_{0,n-1}^{n-1}(x,y).\nonumber
\end{eqnarray}
In the same way, working on \eqref{3try},  we can write $yQ_{k,n-k}^n(x,y),$  for $k = 0,1,\ldots,n$, as  
\begin{align} \label{yQ}
&  yQ_{k,n-k}^n(x,y) = Q_{k,n+1-k}^{n+1}(x,y) \nonumber \\[1ex]
&  + c^{(2)}_{n-k,0}Q_{n,0}^{n}(x,y) + \cdots + c^{(2)}_{n-k,l} Q_{n-l,l}^{n}(x,y) + \cdots + c^{(2)}_{n-k,n}Q_{0,n}^{n}(x,y) 
\\[1ex]
&  +  d^{(2)}_{n-k,0}Q_{n-1,0}^{n-1}(x,y) + \cdots + d^{(2)}_{n-k,s} Q_{n-1-s,s}^{n-1}(x,y) + \cdots + d^{(2)}_{n-k,n-1} Q_{0,n-1}^{n-1}(x,y) \nonumber
\end{align}
Making the change of variables $ x \leftrightarrow y $ in \eqref{yQ} and using Theorem \ref{prop_sim_pol}, we obtain 
\begin{eqnarray}\label{xQII}
xQ_{n-k,k}^n(x,y)\!\!\!&\!\!=\!\!&\!\!\! Q_{n+1-k,k}^{n+1}(x,y) + c^{(2)}_{n-k,0} Q_{0,n}^{n}(x,y) + \cdots + c^{(2)}_{n-k,l} Q_{l,n-l}^{n}(x,y) + \cdots \nonumber  \\[1ex]
\!\!&&\!\!+ c^{(2)}_{n-k,n-l} Q_{n-l,l}^{n}(x,y) + \cdots + c^{(2)}_{n-k,n} Q_{n,0}^{n}(x,y) \\[1ex]
\!\!&&\!\!+ d^{(2)}_{n-k,0} Q_{0,n-1}^{n-1}(x,y) + \cdots + d^{(2)}_{n-k,s} Q_{s,n-1-s}^{n-1}(x,y) + \cdots \nonumber  \\[1ex]
\!\!&&\!\!+  d^{(2)}_{n-k,n-1-s} Q_{n-1-s,s}^{n-1}(x,y) + \cdots + d^{(2)}_{n-k,n-1} Q_{n-1,0}^{n-1} (x,y). \nonumber
\end{eqnarray}
	
Since  $\{Q_{n-l,l}^n\}^{n}_{l=0} \cup\{Q_{n-1-s,s}^{n-1}\}^{n-1}_{s=0}$ is a linearly independent set, comparing \eqref{xQ} and \eqref{xQII}, we obtain	
$$
	c^{(1)}_{ij} = c^{(2)}_{n-i, n-j}, \quad i,j=0,1, \ldots, n,
$$
and
$$
	d^{(1)}_{ij} = d^{(2)}_{n-i, n-1-j},	\quad i=0,1, \ldots,n, \ \ j=0,1, \ldots,n-1.
$$
Hence, $C_{n,1} \leftrightharpoons C_{n,2}$ and $D_{n,1} \leftrightharpoons D_{n,2}$. 
\end{proof}

We remark that, since  $C_{n,1} \leftrightharpoons C_{n,2}$ and $D_{n,1} \leftrightharpoons D_{n,2}$, it is enough to calculate the coefficient matrices associated with one variable and the other coefficient matrices are directly obtained.

\medskip
 
The reciprocal is also true. 

\begin{theorem}
	If the matrices $C_{n,i}$ and $D_{n,i}$, $i=1,2$, in \eqref{M3TR} satisfy $C_{n,1} \leftrightharpoons C_{n,2}$ and $D_{n,1} \leftrightharpoons D_{n,2}$, then the associated MOPS $\{\Qset_{n}\}_{n\geqslant0}$ is reflexive.
\end{theorem}
\begin{proof}
Let	 $C_{n,k}=(c^{(k)}_{ij})_{i,j=0}^{n,n}$, $D_{n,k}=(d^{(k)}_{ij})_{i,j=0}^{n,n-1}$,  $k=1,2,$ and $\Qset_n$ be the monic orthogonal polynomial vector. The three-term relations \eqref{M3TR} for $n=0$, are
$$
x\Qset_0 = L_{0,1}\Qset_1 + C_{0,1}\Qset_0,  
$$
$$
y\Qset_0 = L_{0,2}\Qset_1 + C_{0,2}\Qset_0.
$$
Denoting $\Qset_0= (1)$ and $\Qset_1 = \left(\begin{matrix}
		x+\alpha\\
		y+\beta
	\end{matrix}\right)$, it follows that 
$$
\begin{array}{lll}
x=(1\quad 0)\left(\begin{matrix}
	x+\alpha\\
	y+\beta\end{matrix}\right) + c_{00}^{(1)},   
\\[3ex]
y=(0\quad 1)\left(\begin{matrix}
	x+\alpha\\
	y+\beta\end{matrix}\right) + c_{00}^{(2)}.	
\end{array}
$$
It is easy to see that $\alpha = \beta$, since by hypothesis, $c_{00}^{(1)} = c_{00}^{(2)}$. Hence, $Q^1_{1,0}(x,y) = Q^1_{0,1}(y,x)$, that is, $\Qset_0$ and $\Qset_1$ are reflexive polynomial vectors. 

We now prove by mathematical induction that, if $\Qset_{n}$ and $\Qset_{n-1}$ are reflexive polynomial vectors, then it holds for $\Qset_{n+1}$. Omitting the variables $(x,y)$, the three-term relations (\ref{M3TR}) can be written as
\begin{equation}\label{I}
L_{n,1} \Qset_{n+1} = x \Qset_n - C_{n,1} \Qset_n - D_{n,1} \Qset_{n-1},
\end{equation}
\begin{equation}\label{II}
L_{n,2} \Qset_{n+1} = y \Qset_n - C_{n,2} \Qset_n - D_{n,2} \Qset_{n-1}.
\end{equation}

From \eqref{I}, for  $k = 0,1,\ldots,n,$ we can write $ Q_{n+1-k,k}^{n+1}(x,y)$ as 
$$
 Q_{n+1-k,k}^{n+1}(x,y) = xQ_{n-k,k}^{n}(x,y) - \sum_{j=0}^n c^{(1)}_{kj}Q_{n-j,j}^{n}(x,y) - \sum_{j=0}^{n-1}d^{(1)}_{kj}Q_{n-1-j,j}^{n-1}(x,y).
$$ 
 On the other hand, from \eqref{II}, we can write $Q_{k,n+1-k}^{n+1}(x,y)$, for $k = 0,1,\ldots,n,$ as 
$$
Q_{k,n+1-k}^{n+1}(x,y) = y Q_{k,n-k}^{n}(x,y) - \sum_{j=0}^{n} c^{(2)}_{n-k,j}Q_{n-j,j}^{n}(x,y) - \sum_{j=0}^{n-1}d^{(2)}_{n-k,j}Q_{n-1-j,j}^{n-1}(x,y),
$$ 

Since  $\Qset_{n}$ and $\Qset_{n-1}$ are reflexives, and $C_{n,1} \leftrightharpoons C_{n,2}$ and $D_{n,1} \leftrightharpoons D_{n,2}$, we get
\begin{align*}
Q_{k,n+1-k}^{n+1}(y,x) 
= &  xQ_{k,n-k}^{n}(y,x)-\sum_{j=0}^nc^{(2)}_{n-k,j}Q_{n-j,j}^{n}(y,x)-\sum_{j=0}^{n-1}d^{(2)}_{n-k,j}Q_{n-1-j,j}^{n-1}(y,x)\\
= &  xQ_{n-k,k}^{n}(x,y)-\sum_{j=0}^nc^{(1)}_{k,n-j}Q_{j,n-j}^{n}(x,y)-\sum_{j=0}^{n-1}d^{(1)}_{k,n-j}Q_{j,n-1-j}^{n-1}(x,y)\\
= & xQ_{n-k,k}^{n}(x,y)-\sum_{j=0}^nc^{(1)}_{kj}Q_{n-j,j}^{n}(x,y)-\sum_{j=0}^{n-1}d^{(1)}_{kj}Q_{n-1-j,j}^{n-1}(x,y) \\
= & Q_{n+1-k,k}^{n}(x,y).
\end{align*}
Therefore, $\Qset_{n+1}$ is reflexive polynomial vector.
\end{proof}

In the next result we relate  reflexive MOPS with centrosymmetric matrices. 

\begin{theorem}\label{E_n centrosymmetric matrix}
With the same assumptions of the Theorem \ref{sobre_E_1_E_2}, for $n\in\N$, the following matrices
$$
\widehat{C}_n =	\left( \begin{array}{c|c}
C_{n,1} & \mathtt{0} \\ \hline
\mathtt{0} & C_{n,2}
\end{array} \right)
	\quad	  \mbox{and} \quad
\widehat{D}_n = \left( \begin{array}{c|c}
D_{n,1} & \mathtt{0} \\ \hline
\mathtt{0} & D_{n,2}
\end{array} \right)
$$ 
are centrosymmetric matrices.
\end{theorem}
\begin{proof}
 The proof follows directly from  Proposition \ref{prop_junType2} and  Theorem \ref{sobre_E_1_E_2}.
\end{proof}

\medskip

We now present the results for orthonormal polynomial systems.

We need the following lemma, about  the square root of a symmetric centrosymmetric positive definite matrix. Following \cite[p. 440]{HJ13}, if $X$ is a symmetric positive definite matrix, there exists a unique symmetric positive definite matrix, $X^\frac{1}{2}$, such that $X^\frac{1}{2}X^\frac{1}{2} = X$. The matrix  $X^\frac{1}{2}$ is called the square root matrix of $X$.

\begin{lemma} \label{Lemmasquare}
Let $X$ be a symmetric centrosymmetric positive definite matrix. Then, the square root matrix $X^{\frac{1}{2}}$ is symmetric centrosymmetric positive definite.
\end{lemma}
\begin{proof}  
Considering $X = (x_{ij})_{i,j=0}^n$ a symmetric centrosymmetric positive definite of size $n+1$, there exist an orthogonal matrix $R=(r_{ij})_{i,j=0}^n$ and a diagonal matrix $D = \diag (\lambda_0, \lambda_1,\ldots, \lambda_n),$ where $\lambda_i >0,$ for $ i=0,1, \ldots, n$, are the eigenvalues of $X$, such that 
\begin{equation}\label{H_n}
		X = R \,D \,R^T,
\end{equation}
see \cite{HJ13}. Hence, the square root matrix of $X$ can be written as
\begin{equation}\label{H_n^{1/2}}
		X^\frac{1}{2} = R \,D^\frac{1}{2}\, R^T,
\end{equation}
where $D^\frac{1}{2} = \diag(\sqrt{\lambda_0}, \sqrt{\lambda_1}, \ldots, \sqrt{\lambda_n})$ and we denote  $X^\frac{1}{2} = (\tilde{x}_{ij})_{i,j=0}^n$.

First we suppose that $\lambda_0 = \lambda_1 = \cdots = \lambda_n$, then $X^\frac{1}{2} = \sqrt{\lambda} I_{n+1}$ and it is 
symmetric centrosymmetric positive definite matrix.

We suppose now that the eigenvalues $\lambda_i$, $i=0,1,\ldots,n$, are not simultaneously equals.
From \eqref{H_n} we observe that the entries of matrix $X$ are linear combinations of $\{\lambda_0, \lambda_1,\ldots, \lambda_n\}$, and it is possible to write
$$
x_{ij} = \sum_{k=0}^{n} \,r_{ik}\,r_{jk}\,\lambda_k.
 = \sum_{k=0}^{n} \,d_k^{(i,j)}\, \lambda_k,
$$
where $d_k^{(i,j)} = r_{ik}\,r_{jk}$, for $k=0,1,\ldots, n.$
Similarly, from \eqref{H_n^{1/2}},  the entries of matrix $X^\frac{1}{2}$ can be written as 
$$ 
\tilde{x}_{ij} = \sum_{k=0}^{n} \,r_{ik}\,r_{jk}\,\sqrt{\lambda_k}  = \sum_{k=0}^{n} \,d_k^{(i,j)}\, \sqrt{\lambda_k}.
$$

From Proposition \ref{Prop_SCPD}, we know that the eigenvectors of a symmetric centrosymmetric positive definite matrix are either symmetric or skew-symmetric vectors. In the decomposition $X=R D R^T$ the columns of matrix $R$ correspondent to the eigenvectors of $X$. Hence the entries of the $k$th column of $R$ satisfy either
$$
r_{i,k} = r_{n-i,k}, \quad i=0,1,\ldots,n,
$$
or
$$ 
r_{i,k} = -r_{n-i,k}, \quad i=0,1,\ldots,n.
$$
Therefore, for a fixed $k$, one can conclude the following 
$$
d_k^{(i,j)} = r_{ik} \, r_{jk} = r_{n-i,k} \, r_{n-j,k} = d_k^{(n-i,n-j)}.
$$
Finally, we observe that, for $i,j= 0,1,\ldots, n,$
$$
\tilde{x}_{ij}   = \sum_{k=0}^{n} \,d_k^{(i,j)}\, \sqrt{\lambda_k} = \sum_{k=0}^{n} \,d_k^{(n-i,n-j)}\, \sqrt{\lambda_k} =
\tilde{x}_{n-i,n-j} , 
$$
therefore, $X^\frac{1}{2}$ is symmetric centrosymmetric positive definite matrix. 
\end{proof}

Next results relate reflexive orthonormal polynomial system with centrosymmetric matrices.

\begin{theorem}\label{Teorema Self-Rev}
Consider the orthonormal polynomial system $\{\Pset_n\}_{n\geqslant 0}$ associated with a weight function $W(x,y)$ satisfying \eqref{W(x,y)=W(y,x)} and defined by $\Pset_n = H_n^{-1/2} \Qset_n$, for $n \geqslant 0$, where $\{\Qset_n\}_{n\geqslant 0}$ is the associated MOPS and  $H_n=\langle \mathbf{u}, \Qset_{n}\Qset_{n}^T \rangle$.
Let $A_{n,i}$ and $B_{n,i}$, $i=1,2$, be the coefficient matrices of the three-term relations \eqref{O3TR} for $\{\Pset_n\}_{n\geqslant 0}$. Then, the matrices
\begin{equation*} 
		\widehat{A}_n = \left( \begin{array}{c|c}
			A_{n,1} & \mathtt{0}\\
			\hline
			\mathtt{0} & A_{n,2}
		\end{array} \right)
\quad \mbox{and} \quad
		\widehat{B}_n = \left( \begin{array}{c|c}
			B_{n,1} & \mathtt{0}\\
			\hline
			\mathtt{0} & B_{n,2}
		\end{array} \right)
\end{equation*}
are centrosymmetric.
\end{theorem}
\begin{proof}
From \cite{DX14}, we know that, for orthonormal polynomial system $\{\Pset_n\}_{n\geqslant 0}$, defined by $\Pset_n = H_n^{-1/2} \Qset_n$,  the matrices $A_{n,i}$, $B_{n,i}$, $C_{n,i}$, and $D_{n,i}$ of the associated three-term relations  are related as	
\begin{equation*}
A_{n,i} = H_n^{\frac{1}{2}} D_{n+1,i}^T H_{n+1}^{-\frac{1}{2}}  \quad and \quad 
B_{n,i} = H_n^{-\frac{1}{2}} C_{n,i} H_{n}^{\frac{1}{2}},  \quad i=1,2,
\end{equation*}
where the 
matrix $H_n^{\frac{1}{2}}$ is the symmetric centrosymmetric positive definite matrix such that $H_n^{\frac{1}{2}}H_n^{\frac{1}{2}} = H_n$.
	
Hence, the matrix $\widehat{A}_n$ can be written as
\begin{equation*}
\widehat{A}_n = \left( \begin{array}{c|c}
H_{n}^{\frac{1}{2}} D_{n+1,1}^T H_{n+1}^{-\frac{1}{2}} & \mathtt{0}\\
\hline
\mathtt{0} & H_{n}^{\frac{1}{2}} D_{n+1,2}^T 
H_{n+1}^{-\frac{1}{2}} 
\end{array} \right),
\end{equation*}
or   
$$
\widehat{A}_n = \left( \begin{array}{c|c}
	H_{n}^{\frac{1}{2}}  & \mathtt{0}\\
	\hline
	\mathtt{0} & H_{n}^{\frac{1}{2}}
\end{array} \right)
\left( \begin{array}{c|c}
	D_{n+1,1}^T  & \mathtt{0}\\
	\hline
	\mathtt{0} & D_{n+1,2}^T 
\end{array} \right)
\left( \begin{array}{c|c}
	H_{n+1}^{-\frac{1}{2}} & \mathtt{0}\\
	\hline
	\mathtt{0} &  H_{n+1}^{-\frac{1}{2}} 
\end{array} \right). 
$$

From Corollary \ref{Hn is self} and Lemma \ref{Lemmasquare}, we know that $H_n^{\frac{1}{2}}$ is centrosymmetric matrix, then $H_n^{\frac{1}{2}} \leftrightharpoons H_n^{\frac{1}{2}}$.
Hence, from Proposition \ref{prop_junType2} 
$$
\left( \begin{array}{c|c}
	H_{n}^{\frac{1}{2}}  & \mathtt{0}\\
	\hline
	\mathtt{0} & H_{n}^{\frac{1}{2}}
\end{array} \right)
\quad \mbox{and} \quad 
\left( \begin{array}{c|c}
	H_{n+1}^{-\frac{1}{2}} & \mathtt{0}\\
	\hline
	\mathtt{0} &  H_{n+1}^{-\frac{1}{2}} 
\end{array} \right)	
$$
are centrosymmetric matrices.
Furthermore, since  $D_{n+1,1} \leftrightharpoons D_{n+1,2}$, from Theorem \ref{E_n centrosymmetric matrix}
$$
\left( \begin{array}{c|c}
	D_{n+1,1}^T  & \mathtt{0}\\
	\hline
	\mathtt{0} & D_{n+1,2}^T 
\end{array} \right)
$$
is also centrosymmetric matrix. Therefore, from item 2) in Proposition \ref{prod. centrosymmetric matrices}, the matrix $\widehat{A}_n$ is centrosymmetric. 

Similarly, since 
$B_{n,i} = H_n^{-\frac{1}{2}} C_{n,i} H_{n}^{\frac{1}{2}}, $ $ i=1,2,$
and $C_{n,1} \leftrightharpoons C_{n,2}$, it follows that $\widehat{B}_n$ is centrosymmetric matrix.
\end{proof}

Finally, we show that the orthonormal polynomial system, $\{\Pset_n\}_{n\geqslant 0}$, given by $\Pset_n = H_n^{-1/2} \Qset_n$, associated with a reflexive weight function also has the reflexive property.  

\begin{corollary}
With the same hypotheses of Theorem \ref{Teorema Self-Rev},  the orthonormal polynomial  $\Pset_n$, defined by $\Pset_n = H_n^{-1/2} \Qset_n$, for $n \geqslant 0$, is a reflexive polynomial vector.
Furthermore, the associated coefficient matrices  of the three-term relations \eqref{O3TR} satisfy $A_{n,1}\leftrightharpoons A_{n,2}$ and $B_{n,1}\leftrightharpoons B_{n,2}$.
\end{corollary}
\begin{proof}
From Theorem \ref{prop_sim_pol}, $\Qset_n$ is a reflexive polynomial vector.
From Corollary \ref{Hn is self} and Lemma \ref{Lemmasquare}, the matrix $H_n^{\frac{1}{2}}$ is centrosymmetric. Hence, from Proposition \ref{mud. de base}, it follows that $\Pset_n = H_n^{-1/2} \Qset_n$ is a reflexive polynomial vector.

The reverse relations $A_{n,1}\leftrightharpoons A_{n,2}$ and $B_{n,1}\leftrightharpoons B_{n,2}$ follow  from Proposition \ref{prop_junType2}.
\end{proof}

\section{Particular Cases}

In this section we present several particular cases and examples of reflexive  OPS to illustrate the new concepts.

\subsection{Tensor product of univariate orthogonal polynomials}

The simplest example  of bivariate orthogonal polynomials system is associated with the tensor product of the same weight function in one variable. Consider  $w(x)$ a weight function defined for $x \in (a,b)$, then 
$$
W(x,y) = w(x)w(y) =  W(y,x), \quad for \ (x,y)\in (a,b) \times (a,b),
$$
is a reflexive weight function.
The associated OPS, $\{\Pset_n\}_{n\geqslant 0}$, is given by 
$$
\Pset_n = \left(\begin{matrix}
	p_n(x)p_0(y) \\[2ex]
	p_{n-1}(x)p_1(y) \\[2ex]
	\vdots \\[2ex]
	p_0(x)p_n(y)
\end{matrix}\right),
$$
where $\{p_n(x)\}_{n\geqslant 0}$ is the orthogonal polynomial sequence with respect to $w(x)$. Hence, $\Pset_n$ is a reflexive orthogonal polynomial vector.

Moreover, consider the three-term recurrence relation for the orthogonal polynomials $p_n(x)$, ${n\geqslant 0},$ given by
$$
xp_n(x) = \lambda_n p_{n+1}(x) + \gamma_n p_n(x) + \upsilon_n p_{n-1}(x), \quad n\geqslant 0,
$$ 
with $p_{-1}(x)=0$ and $ \lambda_n, \gamma_n, \upsilon_n \in \mathbb{R}$. It is known that the tree-term relations for the associated OPS are
$$
x_i\Pset_n = \Lambda_{n,i}\Pset_{n+1} + \Gamma_{n,i}\Pset_n + \Upsilon_{n,i}\Pset_{n-1}, \quad i=1,2,
$$
where $x_1=x$, $x_2=y$ and the coefficient matrices  are given by 
$$
\Lambda_{n,1} = \left(\begin{array}{ccccc|c}
\lambda_n &  &  &  & \bigcirc & 0 \\
	      & \lambda_{n-1} &  &  &   & 0 \\
	      &  &  & \ddots &  & \vdots \\
\bigcirc  &  &  &  & \lambda_0 & 0
\end{array}\right), \quad 
\Lambda_{n,2} = \left(\begin{array}{c|ccccc}
0      &   \lambda_0   &   &   &         & \bigcirc \\
0      &   &    \ddots  &       &         &  \\
\vdots &   &        &       & \lambda_{n-1} &  \\
0      &\bigcirc &   &   &         &  \lambda_n 
\end{array}\right),$$
$$
\Gamma_{n,1}=\left(\begin{matrix}
\gamma_n &  &  & \bigcirc \\
	 	& \gamma_{n-1}&  &  \\
	 &  & \ddots &  \\
	\bigcirc &  &  & \gamma_0
\end{matrix}\right), \quad 
\Gamma_{n,2}=\left(\begin{matrix}
\gamma_0 &  &  & \bigcirc \\
 & \ddots &  &  \\
 &  & \gamma_{n-1} &  \\
\bigcirc &  &  & \gamma_n 
\end{matrix}\right), $$
$$
\Upsilon_{n,1} = \left(\begin{matrix}
	\upsilon_n &   &  & \bigcirc \\
	 & \upsilon_{n-1} &  &  \\
	 &  & \ddots  &  \\
	\bigcirc &  &  & \upsilon_1 \\
	\hline
	0 & 0 & \ldots & 0
\end{matrix}\right), \quad \mbox{and} \quad 
\Upsilon_{n,2} = \left(\begin{matrix}
0 & \ldots & 0 & 0 \\
\hline 
\upsilon_1 &  &  & \bigcirc \\
 & \ddots &  &  \\
 &  & \upsilon_{n-1} &  \\
\bigcirc &  &  & \upsilon_n 
\end{matrix}\right).
$$
Clearly, 
$\Lambda_{n,1} \leftrightharpoons \Lambda_{n,2},$
$\Gamma_{n,1} \leftrightharpoons \Gamma_{n,2}$,  and 
$\Upsilon_{n,1} \leftrightharpoons \Upsilon_{n,2}$. 
 
\subsection{Bivariate orthogonal polynomials on the simplex}

	Let us consider the family of the monic bivariate orthogonal polynomials on the simplex $\Omega=\{(x,y)\in\R^2 \, | \, x,y\geqslant 0,1-x-y\geqslant 0\}$	
associated with the weight function
$$
W^{(\alpha,\beta,\gamma)}(x,y) = x^\alpha y^\beta(1-x-y)^\gamma, \quad \alpha, \beta, \gamma > -1.
$$

We consider $\alpha = \beta$ and the MOPS associated with the weight function 
\begin{equation}\label{simplex_wf}  W_{\alpha,\alpha,\gamma}(x,y) = (xy)^\alpha (1-x-y)^\gamma, \quad \alpha,  \gamma > -1.
\end{equation}
Observe that $ W_{\alpha,\alpha,\gamma}(x,y)$ defined in $\Omega$ is a reflexive weight function. Following \cite[p.36]{DX14}, let 
$\{\mathbb{V}^{(\alpha,\gamma)}_n\}_{n\geqslant0}$ be the monic OPS on the simplex orthogonal with respect to the weight function $W_{\alpha,\alpha,\gamma}(x,y)$.

From the explicit formula given in \cite[p. 36]{DX14}, we can calculate the first vectors of the MOPS, that are 
$$
\mathbb{V}^{(\alpha,\gamma)}_0 = (1), \quad 
\mathbb{V}^{(\alpha,\gamma)}_1 = \begin{pmatrix}
		x - \dfrac{2\alpha+1}{4\alpha + 2\gamma + 3}\\[2ex]
		y - \dfrac{2\alpha+1}{4\alpha + 2\gamma + 3}	
\end{pmatrix}, 
$$
and
$$
\mathbb{V}^{(\alpha,\gamma)}_2 = \begin{pmatrix}
x^2 - \dfrac{2(2\alpha + 3)}{4\alpha + 2\gamma + 7}x + \dfrac{(2\alpha+1)(2\alpha+3)}{(4\alpha + 2\gamma + 5)
(4\alpha + 2\gamma + 7)} \\[3ex]
xy - \dfrac{2\alpha+1}{4\alpha+2\gamma+7}(x+y)+\dfrac{(2\alpha+1)^2}{(4\alpha+2\gamma+5)(4\alpha+2\gamma+7)} \\[3ex]
y^2 - \dfrac{2(2\alpha + 3)}{4\alpha + 2\gamma + 7}y + \dfrac{(2\alpha+1)(2\alpha+3)}{(4\alpha + 2\gamma + 5)
(4\alpha + 2\gamma + 7)}
\end{pmatrix}.
$$

These polynomial vectors are reflexives. Now, consider the three-term relation satisfied by $\{\mathbb{V}^{(\alpha,\gamma)}_n\}_{n\geqslant0}$, 
$$
x_i\mathbb{V}^{(\alpha,\gamma)}_n = L_{n,i}\mathbb{V}^{(\alpha,\gamma)}_{n+1} + 
C^{(\alpha,\gamma)}_{n,i}\mathbb{V}^{(\alpha,\gamma)}_n 
+ D^{(\alpha, \gamma)}_{n,i}\mathbb{V}^{(\alpha,\gamma)}_{n-1}, \quad i=1,2. 
$$
From \cite{AAP22}, the shape of the coefficient matrices is 
$$
C^{(\alpha, \gamma)}_{n,1} =\begin{pmatrix}
	c^{(1)}_{00} & \quad & \quad & \bigcirc \\
	c^{(1)}_{10} & c^{(1)}_{11} & \quad & \quad \\
	\quad & \ddots & \ddots & \quad \\
	\bigcirc & \quad & c^{(1)}_{n,n-1} & c^{(1)}_{nn} 
\end{pmatrix}, 
C^{(\alpha, \gamma)}_{n,2} = \begin{pmatrix}
	c^{(2)}_{00} & c^{(2)}_{01} & \quad & \bigcirc \\
	\quad & c^{(2)}_{11} & \ddots & \quad \\
	\quad & \quad & \ddots & c_{n-1,n}^{(2)} \\
	\bigcirc & \quad & \quad & c^{(2)}_{nn} 
\end{pmatrix},
$$
$$
D^{(\alpha,  \gamma)}_{n,1}=\begin{pmatrix}
	d^{(1)}_{00} & \quad & \quad   & \quad             & \bigcirc \\[1ex]
	d^{(1)}_{10} & d^{(1)}_{11} & \quad           & \quad & \quad \\[1ex]
	d^{(1)}_{20} & d^{(1)}_{21} & d^{(1)}_{22} & \quad & \quad \\[1ex]
	\quad           & \ddots & \ddots & \ddots & \quad \\[1ex]
	\quad           & \quad  & \ddots & \ddots & d^{(1)}_{n-1,n-1} \\[1ex]
	\bigcirc        & \quad  & \quad  & d^{(1)}_{n,n-2} & d^{(1)}_{n,n-1}
\end{pmatrix}
$$
and
$$ D^{(\alpha,  \gamma)}_{n,2}=\begin{pmatrix}
	d^{(2)}_{00} & d^{(2)}_{01} & \quad & \quad & \bigcirc \\[1ex]
	d^{(2)}_{10} & d^{(2)}_{11} & d^{(2)}_{12} & \quad & \quad \\[1ex]
	\quad & d^{(2)}_{2,1} & d^{(2)}_{22} & \ddots & \quad \\[1ex]
	\quad & \quad & \ddots & \ddots & d^{(2)}_{n-2,n-1} \\[1ex]
	\quad & \quad & \quad & d^{(2)}_{n-1,n-2}  & d^{(2)}_{n-1,n-1} \\[1ex]
	\bigcirc & \quad & \quad & \quad & d^{(2)}_{n,n-1}
\end{pmatrix}. 
$$
The explicit expressions of the entries are
\begin{align*}
c^{(1)}_{ii} & = \dfrac{(i-n)(\alpha-i+n)}{2\alpha + \gamma + 2n + 1} + \dfrac{(-i+n+1)(\alpha -i +n+1)}{2\alpha +\gamma + 2n + 3}\\[1ex]	
& = \dfrac{(n-i+1)(\alpha+n-i+1)}{2\alpha + \gamma + 2n + 3} - \dfrac{(n-i)(\alpha + n-i)}{2\alpha +\gamma + 2n + 1} = c^{(2)}_{n-i,n-i}, \quad 0 \leqslant i \leqslant n,
\\
c^{(1)}_{i+1,i} & = - \dfrac{2(i+1)(\alpha+i+1)}{(2\alpha + \gamma + 2n + 1)(2\alpha+\gamma+2n+3)}\\[1ex] 
& = \dfrac{2(n-i-1-n)(\alpha-n+i+1+n)}{(2\alpha + \gamma + 2n + 1)(2\alpha+\gamma+2n+3)} = c^{(2)}_{n-(i+1),n-i}, \quad 0 \leqslant i \leqslant n-1,
	\end{align*}
and $c^{(1)}_{ij}=0$, for $j>i$ or $j<i-1$, $c^{(2)}_{ij}=0$, for $j<i$ or $j>i+1$, it follows that $C^{(\alpha,  \gamma)}_{n,1}\leftrightharpoons C^{(\alpha,  \gamma)}_{n,2}$. 

Analogously, working with the explicit expressions of the entries of the matrices $D^{(\alpha,  \gamma)}_{n,i}$, $i=1,2$, given in \cite{AAP22}, we get
\begin{align*}
& d^{(1)}_{ii}  = d^{(2)}_{n-i,n-1-i}, \quad 0\leqslant i \leqslant n-1, \\[.6ex]
& d^{(1)}_{i+1,i} = d^{(2)}_{n-i-1,n-1-i},\quad 0 \leqslant i \leqslant n-2, \\[.6ex]
& d^{(1)}_{i+2,i}  =  d^{(2)}_{n-i-2,n-1-i}, \quad 0\leqslant i \leqslant n-2,
\end{align*}
and, since $d^{(1)}_{ij}=0$, for $j>i$ or $j<i-2$, $d^{(2)}_{ij}=0$, for $j<i-1$ or $j>i+1$, it follows that $D^{(\alpha,  \gamma)}_{n,1}\leftrightharpoons D^{(\alpha,  \gamma)}_{n,2}$.

\subsection{Bivariate Freud weight function}

In \cite{BCP22} we investigate the
bivariate Freud weight function given by
$$
W(x,y) =e^{-q(x,y)},\qquad (x,y) \in \mathbb{R}^2,
$$
where 
$$
q(x,y) = a_{4,0} \, x^4 + a_{2,2} \, x^2 \, y^2 + a_{0,4}\, y^4 + a_{2,0} \, x^2 + a_{0,2} \, y^2
$$
and $a_{i,j}$ are real parameters. 
Setting $a_{4,0}=a_{0,4}=a$, $a_{2,2}=b$, and $a_{2,0}=a_{0,2}=c$, the particular case,
$$
W(x,y) = e^{-[a (x^4 +y^4) + b \,x^2 \, y^2   + c(x^2+y^2)]}
$$
is a reflexive weight function. Since $W(x,y)$ is an even function, the MOPS  $\{\mathbb{Q}_n\}_{n\geqslant 0}$ satisfy the three-term relations
\begin{equation*}
		x_i\Qset_n = L_{n,i}\Qset_{n+1} + D_{n,i}\Qset_{n-1}, \quad i=1,2,
\end{equation*}
for $n\geqslant 0$, $\mathbb{Q}_{-1} =0$ and $\mathbb{Q}_0 = (1)$, where  $D_{n,1}\leftrightharpoons D_{n,2}$.

\subsection{Uvarov modification}

Let $W(x,y)$ be a weight function defined on a domain $\Omega\subset\R^2$ such that all moments exist
$$
\iint_\Omega x^k\,y^l\, W(x,y)dxdy < +\infty,
$$ 
for $k, l \geqslant 0$. 
Define the inner products 
$$ 
(f, g) = \iint_\Omega f(x,y) g(x,y) W(x,y) dxdy, 
$$
and 
$$ 
(f, g)_U = (f, g) + \mathbf{f(x)} \, \Lambda  \,\mathbf{g(x)}^T, 
$$
where $ \mathbf{x} = ((x_0,y_0), (x_1,y_1), \ldots , (x_n,y_n))$, $(x_i,y_i) \in \mathbb{R}^2$, $ i=0,1, \ldots , n$, are fixed points, $\Lambda$ is a symmetric positive semi-definite matrix of size $ (n+1) \times (n+1) $,  $ \mathbf{f(x)} = (f(x_0,y_0), $ $f(x_1,y_1), \ldots, f(x_n,y_n))$, and $ \mathbf{g(x)} = (g(x_0,y_0), g(x_1,y_1), \ldots, g(x_n,y_n))$. This modification of the original inner product is known as Uvarov modification, see \cite{AAP22,DFPPY10}.

Now, we show that if $ W(x,y) $ is reflexive weight function, $ \Lambda $ is centrosymmetric matrix, and $\mathbf{x}=((x_0,y_0), (x_1,y_1), \ldots , (x_n,y_n))$ is such that $ x_i = y_{n-i},$ $ i=0,1, \ldots , n$, then the inner product $(f, g)_U$ is reflexive in the sense that
$$ 
(x^k, y^l)_U = (x^l, y^k)_U, \quad k, l \geqslant 0.
$$
In fact, let  $\Lambda =(\lambda_{ij})_{i,j=0}^{n,n}$, and
$$
(x^l, y^k)_U = \displaystyle \iint_\Omega x^ly^k W(x,y) dx dy + \mathbf{x^l} \Lambda (\mathbf{y^k})^T, 
$$
where $\mathbf{x^l} = (x_0^l, x_1^l,  \ldots, x_n^l)$ and  $\mathbf{y^k} = (y_0^k, y_1^k,  \ldots, y_n^k)$. Then, 
\begin{align*}
(x^l, y^k)_U &
= \iint_\Omega x^l y^k W(x,y) dx dy + \displaystyle \sum_{i=0}^n \sum_{j=0}^n \lambda_{ij}x_i^l y_j^k\\
& =
\iint_\Omega y^lx^k W(y,x) dx dy + \displaystyle \sum_{i=0}^n \sum_{j=0}^n \lambda_{ij}x_i^l y_j^k\\
&=   \iint_\Omega x^ky^l W(x,y) dx dy + \displaystyle \sum_{i=0}^n \sum_{j=0}^n \lambda_{ij}x_i^l y_j^k\\
& =   \iint_\Omega x^ky^l W(x,y) dx dy + \displaystyle \sum_{i=0}^n \sum_{j=0}^n \lambda_{n-i,n-j}x_{n-j}^k y_{n-i}^l,
\end{align*}
making the change of variables $x \leftrightarrow y$ in the first above integral, using the fact that $W$ is reflexive, $\Lambda$ is centrosymmetric matrix and   $ x_i = y_{n-i}, \, i = 0,1, \ldots, n $.
Hence, making $r=n-i, \, s=n-j$ and using that $ \Lambda$ is symmetric, it becomes 
$$ (x^l, y^k)_U = 
\iint_\Omega x^ky^l W(x,y) dx dy + \displaystyle \sum_{s=0}^n \sum_{r=0}^n \lambda_{s,r}x_{s}^k y_{r}^l = (x^k, y^l)_U.  
$$

\medskip

As a numerical example, we consider again OPS on simplex in $\mathbb{R}^2$, with the weight function $W^{(1,1/2)}(x,y)$ defined in \eqref{simplex_wf}. 
Let $\Lambda = \frac{1}{2} I_3$ and $ \mathbf{x} = ((1,0),(0,0), $ $(0,1))$, hence, the reflexive Uvarov inner product is given by
\begin{align*}
(f,g)_U = & \iint_\Omega \,f(x,y)\,g(x,y)\, x\,y\,\sqrt{1-x-y}\, dx\, dy \\
& + \dfrac{1}{2}[f(1,0)g(1,0) + f(0,0)g(0,0) + f(0,1)g(0,1)]. 
\end{align*}
The first polynomial vectors of the associated MOPS  were calculated in \cite{AAP22}, and they are reflexives:
$$
\Qset_0 = (1), \quad 
\Qset_1 = \begin{pmatrix} x - \dfrac{10459}{31361} \\[2ex]
y - \dfrac{10459}{31361} 
\end{pmatrix}, $$
and
$$
\Qset_2 = \begin{pmatrix}
	x^2 - \dfrac{320811709991693}{321113175737485}x + \dfrac{36006461568}{64222635147497}y + \dfrac{51957376}{30832708855} \\[2ex]
	xy - \dfrac{5355008}{7115240505} (x+y) - \dfrac{69058048}{92498126565}  \\[2ex]
	y^2 + \dfrac{36006461568}{64222635147497}x - \dfrac{320811709991693}{321113175737485}y + \dfrac{51957376}{30832708855} 
\end{pmatrix}.$$

\subsection{Christoffel modification}

Let $\lambda(x,y)$ be a real bivariate polynomial given as
\begin{equation} \label{lambda pol.}
\lambda (x,y) = a \,(x^2 + y^2) + b\, x\,y + c\, (x + y) + d,
\end{equation}
where $ a,b,c,d \in \R $ and $ |a| + |b| > 0 $. We  consider the functional $\mathbf{v}$ obtained from a polynomial modification of a reflexive functional $\mathbf{u}$, given by
$$ 
\langle \mathbf{v}, p(x,y) \rangle = \langle \lambda(x,y) \mathbf{u}, p(x,y) \rangle = \langle \mathbf{u}, \lambda (x,y) \,p(x,y) \rangle,
$$
for $p(x,y) \in \Pi$. The Christoffel modification on a multivariate functional, using multiplication by a polynomial of degree 2, was studied in \cite{DFPP16}. 

Let ${v}_{mn}$ be the associated  moments of the functional $\mathbf{v}.$
Hence,
\begin{align*}
{v}_{m,n} 
= & \ \langle \mathbf{v}, x^my^n \rangle \\
=  & \ \langle \mathbf{u}, [a (x^2 + y^2) + b\, x\,y + c (x + y) + d] x^m y^n  \rangle \\[1ex]
= & \ a \langle \mathbf{u}, x^{m+2}y^n + x^{m}y^{n+2} \rangle
+ b \langle \mathbf{u}, x^{m+1}y^{n+1} \rangle  \\
& \ + c \langle \mathbf{u}, x^{m+1}y^n  +  x^{m}y^{n+1} \rangle + d \langle \mathbf{u}, x^{m}y^n \rangle
\\[1ex]
= & \ a \langle \mathbf{u}, x^{n}y^{m+2} + x^{n+2}y^{m} \rangle
+ b \langle \mathbf{u}, x^{n+1}y^{m+1} \rangle  \\
& \ + c \langle \mathbf{u}, x^{n}y^{m+1}  +  x^{n+1}y^{m} \rangle + d \langle \mathbf{u}, x^n y^m \rangle \\[1ex]
= & \ \langle \mathbf{u}, [a (x^2 + y^2) + b\, x\,y + c (x +  y) + d ] x^n  y^m  \rangle  
 =  {v}_{n,m},
\end{align*}
since $ \mathbf{u} $ is reflexive.  Therefore, the Christoffel modification of a reflexive moment functional by a polynomial of type \eqref{lambda pol.} preserves the reflexive property. 

Let $\{\Qset_n\}_{n\geqslant 0}$ and $\{\widetilde{\Qset}_n\}_{n\geqslant 0}$ be the respective MOPS associated respectively with $\mathbf{u}$ and $\mathbf{v}$. We remark that, since $\mathbf{u}$ and $\mathbf{v}$ are reflexive moment functional, then $\{\Qset_n\}_{n\geqslant 0}$ and $\{\widetilde{\Qset}_n\}_{n\geqslant 0}$ are reflexive MOPS. 

From \cite[Th.~4.1]{DFPP16}, it is known that, for $n \geqslant 1$, there exist real matrices $R_n$ and $S_n$ of respective sizes $(n+1) \times n$ and $ n\times (n-1)$, with $S_2 \not\equiv \mathtt{0}$, such that 
\begin{equation} \label{rel. P and tildeP}
	 \Qset_n = \widetilde{\Qset}_n + R_n \widetilde{\Qset}_{n-1} + S_n \widetilde{\Qset}_{n-2}, \quad n \geqslant 1.
\end{equation}

We now verify that the matrices $R_n$ and $S_n$ are centrosymmetric. First we denote 
$R_{n}=(r^{(n)}_{ij})_{i,j=0}^{n,n-1}$,
$S_{n}=(s^{(n)}_{ij})_{i,j=0}^{n,n-2}$,
$$
\Qset_n = (Q_{n,0}^{n}(x,y), Q_{n-1,1}^{n}(x,y), \ldots, Q_{0,n}^{n}(x,y))^T,
$$
and
$$
\widetilde{\Qset}_n = (\widetilde{Q}_{n,0}^{n}(x,y), \widetilde{Q}_{n-1,1}^{n}(x,y), \ldots, \widetilde{Q}_{0,n}^{n}(x,y))^T.
$$

From relation \eqref{rel. P and tildeP} we get 
\begin{equation} \label{expr n-k,k}
Q_{n-k,k}^{n}(x,y) = \widetilde{Q}_{n-k,k}^{n}(x,y) + \sum_{j=0}^{n-1} r_{n-k,j}^{(n)} \widetilde{Q}_{n-1-j,j}^{n-1}(x,y) + \sum_{j=0}^{n-2} s^{(n)}_{n-k,j} \widetilde{Q}_{n-2-j,j}^{n-2}(x,y),
\end{equation}
for $ k=0,1,\ldots, n,$ and
\begin{equation} \label{expr k,n-k}
{Q}_{k,n-k}^{n}(y,x) = \widetilde{Q}_{k,n-k}^{n}(y,x) + \sum_{j=0}^{n-1} r_{kj}^{(n)} \widetilde{Q}_{n-1-j,j}^{n-1}(y,x) + \sum_{j=0}^{n-2} s^{(n)}_{kj} \widetilde{Q}_{n-2-j,j}^{n-2} (y,x).
\end{equation} 
Using the fact that  $\{\widetilde{\Qset}_n\}_{n \geqslant 0}$ is reflexive, the representation \eqref{expr k,n-k} becomes
\begin{equation*} 
	{Q}_{k,n-k}^{n}(y,x) = \widetilde{Q}_{n-k,k}^{n}(x,y) + \sum_{j=0}^{n-1} r_{kj}^{(n)} \widetilde{Q}_{j,n-1-j}^{n-1}(x,y) + \sum_{j=0}^{n-2} s^{(n)}_{kj} \widetilde{Q}_{j,n-2-j}^{n-2} (x,y).
\end{equation*} 
Making $ l = n-1-j$  in the first summation and $l = n-2-j $ in the second, we get
\begin{equation} \label{expr final}
Q_{k,n-k}^{n}(y,x) = \widetilde{Q}_{n-k,k}^{n}(x,y) + \sum_{l=0}^{n-1} r_{k,n-1-l}^{(n)} \widetilde{Q}_{n-1-l,l}^{n-1}(x,y) + \sum_{l=0}^{n-2} s^{(n)}_{k,n-2-l} \widetilde{Q}_{n-2-l,l}^{n-2} (x,y)
\end{equation} 
for $k=0,1,\ldots, n.$ 

Comparing expressions \eqref{expr n-k,k} and \eqref{expr final}, and using that $\Qset_n$ is reflexive polynomial vector, we have $r^{(n)}_{n-k,j} = r^{(n)}_{k,n-1-j}$ and $s^{(n)}_{n-k,j} = s^{(n)}_{k,n-2-j}$. Finally, setting $i = n-k$, then
$$r^{(n)}_{ij} = r^{(n)}_{k,n-1-j} \quad \mbox{and} \quad s^{(n)}_{ij} = s^{(n)}_{k,n-2-j}.$$
Therefore, $R_n$ and $S_n$, $n \geqslant 1$, are centrosymmetric matrices.

Conversely, from \cite[Th.~4.3]{DFPP16}, consider $\{\Qset_n\}_{n \geqslant 0}$ a MOPS associated with a moment functional $\mathbf{u}$ and two sequences of matrices $\{R_n\}_{n\geqslant 1}$ and $\{S_n\}_{n\geqslant 2}$ of size $(n+1)\times n$ and $(n+1) \times (n-1)$ respectively, with $S_2 \not\equiv \mathtt{0}$. Then, the monic polynomial system, $\{\widetilde{\Qset}_n\}_{n \geqslant 0}$, defined by 
$\widetilde{\Qset}_0  = \Qset_0, $
\begin{align*}
\widetilde{\Qset}_1 & = \Qset_1 - R_1 \Qset_0, \\
\widetilde{\Qset}_n & = \Qset_n - R_n \Qset_{n-1} - S_n \Qset_{n-2}, \quad n \geqslant 2,
\end{align*}
is a MOPS associated with a moment functional  $\mathbf{v}= \lambda(x,y) \mathbf{v}$, where
$$ 
\lambda(x,y) = S_2^TH_2^{-1}\Qset_2 + M_1^TH_1^{-1}\Qset_1 + H_0^{-1}\Qset_0 
$$
and $H_n = \langle \mathbf{u}, \Qset_n\Qset_n^T \rangle.$

Now we consider the particular case when $\mathbf{u}$ is a reflexive moment functional and $\{\Qset_n\}_{n \geqslant 0}$ is a reflexive MOPS. If the matrices $R_n$, ${n\geqslant 1}$ and $S_n,$ $n \geqslant 2,$ are  centrosymmetric matrices, from Proposition \ref{mud. de base} follows that the  MOPS $\{\widetilde{\Qset}_n\}_{n \geqslant 0}$ is also a reflexive MOPS. 

Moreover, from Corollary \ref{Hn is self} and items 3) and 5) of Proposition \ref{prod. centrosymmetric matrices}, the polynomial $\lambda(x,y)$ is a reflexive polynomial. Therefore, $\mathbf{v} = \lambda(x,y) \mathbf{u}$ is a reflexive moment functional.

\section*{Acknowledgements}

This work is part of the doctoral thesis of the second author (GSC) at UNESP, S\~{a}o Jos\'{e} do Rio Preto, SP, Brazil, and it began to be developed during a research visit by GSC to the University of Granada, Spain, supported by the grant 88887.716711/2022-00 from CAPES, in the scope of the Program CAPES-PrInt,  International Cooperation Project number  88887.310463/2018-00, Brazil.

First author (CFB) thanks for the supported, through the Grant 2022/09575-5, from FAPESP, S\~{a}o Paulo, Brazil.

Third author (TEP) thanks Grant CEX2020-001105-M funded by MCIN/AEI/ 10.13039/501100011033 and Research Group Goya-384, Spain.


\end{document}